\documentclass[a4paper, 12pt, oneside]{article}

\usepackage[utf8]{inputenc}
\usepackage{cmap}

\usepackage[english]{babel}
\usepackage[warn]{mathtext}
\usepackage[T2A]{fontenc}
\usepackage{microtype}

\usepackage[a4paper, margin=2cm]{geometry}

\usepackage[unicode=true, colorlinks=true, linkcolor=blue, citecolor=Green]{hyperref}
\usepackage[usenames,dvipsnames]{color}

\usepackage{amssymb,amsfonts,amsmath,amsthm,mathtools}

\usepackage{mathrsfs}

\usepackage[inline]{enumitem}
\usepackage{indentfirst}
\usepackage{titlesec}

\usepackage{comment}

\usepackage{cite}

\usepackage{tikz}

\renewcommand{\le}{\leqslant}
\renewcommand{\ge}{\geqslant}

\numberwithin{equation}{section}

\DeclareMathOperator{\Dom}{Dom}
\DeclareMathOperator{\Ran}{Ran}
\DeclareMathOperator{\spec}{spec}

\titleformat{\part}[hang]{ \large\bfseries}{Глава~\thepart.}{0.5ex}{\centering}[]
\titleformat{\section}[hang]{\large\bfseries}{\thesection.}{0.5ex}{}[]
\titleformat{\subsection}[runin]{\bfseries}{\thesubsection.}{0.5ex}{}[]

\theoremstyle{plain}
\newtheorem{thrm}{\scshape Theorem}

\newtheorem{lemma}[thrm]{\scshape Lemma}

\newtheorem{remark}[thrm]{\scshape Remark}
\newtheorem{condition}[thrm]{\scshape Condition}

\theoremstyle{definition}

\newtheoremstyle{break}
{}{}
{\itshape}{}
{\bfseries}{}
{\newline}{}
\theoremstyle{break}

\theoremstyle{break}

\providecommand{\keywords}[1]{\textbf{{Keywords:}} #1.}

\title{High-frequency homogenization of nonstationary \\ periodic equations\footnote{Supported by Young Russian Mathematics award and Ministry of Science and Higher Education of the Russian Federation, agreement \textnumero\,075-15-2019-1619.}}
\date{}
\author{Mark Dorodnyi\footnote{Leonhard Euler International Mathematical Institute, St.~Petersburg State University, 14th Line V.O., 29B, St.~Petersburg, 199178, Russia; e-mail: \texttt{mdorodni@yandex.ru}.}}

\begin{document}

\clubpenalty = 10000
\widowpenalty = 10000

\hyphenation{мат-ри-чно-знач-ная}
\hyphenation{пе-рио-ди-чес-кие}
\hyphenation{пе-рио-ди-чес-кую}
\hyphenation{пе-рио-ди-чес-кое}

\maketitle

\begin{abstract}
	\noindent We consider an elliptic differential operator $A_\varepsilon = - \frac{d}{dx} g(x/\varepsilon) \frac{d}{dx} + \varepsilon^{-2} V(x/\varepsilon)$, $\varepsilon > 0$, with periodic coefficients acting in $L_2(\mathbb{R})$. For the nonstationary Schr\"{o}dinger equation with the Hamiltonian $A_\varepsilon$ and for the hyperbolic equation with the operator $A_\varepsilon$, analogs of homogenization problems, related to the edges of the spectral bands of the operator $A_\varepsilon$, are studied  (the so called high-frequency homogenization). For the solutions of the Cauchy problems for these equations with special initial data, approximations in $L_2(\mathbb{R})$-norm for small $\varepsilon$ are obtained.
\end{abstract}


\allowdisplaybreaks

\section{Introduction}

\subsection{Periodic homogenization.}
The study of the wave propagation in periodic structures is of significant interest both for applications and from the
theoretical point of view. Direct numerical simulations of such processes may be difficult. One of the approaches to study these problems is application of homogenization theory. The aim of homogenization is to describe the macroscopic properties of inhomogeneous media by taking into account the properties of the microscopic structure. An extensive literature is devoted to homogenization problems. First of all, we mention the books~\cite{BaPa, BeLP, ZhKO}.

Let us discuss a typical problem of homogenization theory. Let $\Gamma$ be a lattice in $\mathbb{R}^d$, and let $\Omega$ be the cell of $\Gamma$. For any $\Gamma$-periodic function $F(\mathbf{x})$, we denote $F^\varepsilon (\mathbf{x}) \coloneqq F(\varepsilon^{-1} \mathbf{x})$, where $\varepsilon > 0$ is a (small) parameter. In $L_2(\mathbb{R}^d)$,  consider a differential operator~(DO) formally given by
\begin{equation}
	\label{intro_hatA_eps}
	\widehat{\mathcal{A}}_\varepsilon = - \operatorname{div} g^\varepsilon (\mathbf{x}) \nabla,
\end{equation}
where $g(\mathbf{x})$ is a Hermitian $\Gamma$\nobreakdash-periodic $(d \times d)$\nobreakdash-matrix-valued function, bounded and positive definite. Operator~\eqref{intro_hatA_eps} models the simplest cases of microinhomogeneous media with $\varepsilon \Gamma$-periodic structure. Let $u_\varepsilon(\mathbf{x})$ be a (weak) solution of the elliptic equation
\begin{equation}
	\label{intro_elliptic_eq}
	- \operatorname{div} g^\varepsilon(\mathbf{x}) \nabla u_\varepsilon(\mathbf{x}) + u_\varepsilon(\mathbf{x}) = f(\mathbf{x}),
\end{equation} 
where $f \in L_2(\mathbb{R}^d)$. For $\varepsilon \to 0$, the solution $u_\varepsilon$ converges to the solution $u_0$ of the "homogenized" equation:
\begin{equation}
	\label{intro_elliptic_eq_homog}
	- \operatorname{div} g^0 \nabla u_0(\mathbf{x}) + u_0(\mathbf{x}) = f(\mathbf{x}).
\end{equation} 
The operator $\widehat{\mathcal{A}}^{\mathrm{hom}}= - \operatorname{div} g^0 \nabla$ is called the effective operator for $\widehat{\mathcal{A}}_\varepsilon$. The matrix $g^0$ is determined by a well-known procedure (see., e.g.,~\cite[Chapter~2, \S\,3]{BaPa}, \cite[Chapter~3, \S\,1]{BSu2003}) that requires solving an auxiliary boundary value problem on the cell $\Omega$. Besides finding the effective coefficients, the following questions are of great interest. \emph{What is the type of convergence $u_\varepsilon \to u_0$? What is an estimate for $u_\varepsilon - u_0$?}

\subsection{Operator error estimates in homogenization.}
M.~Birman and T.~Suslina (see~\cite{BSu2003}) suggested the operator-theoretic (spectral) approach to homogenization problems in $\mathbb{R}^d$, based on the scaling transformation, the Floquet--Bloch theory, and the analytic perturbation theory. 

Let $u_\varepsilon$ be the solution of equation~\eqref{intro_elliptic_eq}, and let $u_0$ be the solution of equation~\eqref{intro_elliptic_eq_homog}. In~\cite{BSu2003},  it was proved that
\begin{equation}
	\label{intro_BSu2003_est_eq}
	\| u_\varepsilon - u_0 \|_{L_2(\mathbb{R}^d)} \le C \varepsilon \|f\|_{L_2(\mathbb{R}^d)}.
\end{equation}
Since $u_\varepsilon = (\widehat{\mathcal{A}}_\varepsilon + I)^{-1} f$ and $u_0 = (\widehat{\mathcal{A}}^{\mathrm{hom}}+ I)^{-1} f$, estimate~\eqref{intro_BSu2003_est_eq} can be rewritten in operator terms:
\begin{equation}
	\label{intro_BSu2003_est}
	\| (\widehat{\mathcal{A}}_\varepsilon + I)^{-1} - (\widehat{\mathcal{A}}^{\mathrm{hom}}+ I)^{-1} \|_{L_2(\mathbb{R}^d) \to L_2(\mathbb{R}^d)} \le C \varepsilon.
\end{equation} 
Parabolic equations were studied in~\cite{Su2004, Su2007}. In operator terms, the following approximation for the parabolic semigroup $e^{-t \widehat{\mathcal{A}}_\varepsilon}$, $t > 0$, was obtained:
\begin{equation}
	\label{intro_Su2004_est}
	\| e^{-t \widehat{\mathcal{A}}_{\varepsilon}} - e^{-t \widehat{\mathcal{A}}^{\mathrm{hom}}} \|_{L_2(\mathbb{R}^d) \to L_2(\mathbb{R}^d)} \le C \varepsilon (t + \varepsilon^2)^{-1/2}, \qquad t > 0.
\end{equation}
Estimates~\eqref{intro_BSu2003_est}, \eqref{intro_Su2004_est} are order-sharp; the constants $C$ are controlled explicitly in terms of the problem data. These estimates are called~\emph{operator error estimates} in homogenization. More accurate approximations for the resolvent and the exponential  with correctors taken into account were found in~\cite{BSu2005, BSu2006, V, Su2010}.

A different approach to operator error estimates (the so called “shift method”) for the elliptic and parabolic problems was suggested by V.~Zhikov and S.~Pastukhova in the papers~\cite{Zh2006, ZhPas2005, ZhPas2006}.  See also the survey~\cite{ZhPas2016}.

The situation with homogenization of nonstationary Schr\"{o}dinger-type equations and hyperbolic equations is quite different. The papers~\cite{BSu2008, Su2017, DSu2018, M2021, DSu2020, D2021} were devoted to such problems. In operator terms, the behavior of the operator-functions $e^{-i t \widehat{\mathcal{A}}_{\varepsilon}}$ and $\cos(t \widehat{\mathcal{A}}_{\varepsilon}^{1/2})$, $\widehat{\mathcal{A}}_{\varepsilon}^{-1/2} \sin(t \widehat{\mathcal{A}}_{\varepsilon}^{1/2})$ (where $\tau \in \mathbb{R}$) for small $\varepsilon$ was studied. For these operator-functions, it is impossible to obtain approximations in the operator norm on $L_2 (\mathbb{R}^d)$, and we are forced to consider the norm of operators acting from the Sobolev space $H^q (\mathbb{R}^d)$ (with a suitable $q$) to $L_2 (\mathbb{R}^d)$. In~\cite{BSu2008}, the following sharp-order estimates were proved:
\begin{align}
	\label{intro_exp_est}
	\| e^{-i t \widehat{\mathcal{A}}_{\varepsilon}} - e^{-i t \widehat{\mathcal{A}}^{\mathrm{hom}}} \|_{H^3 (\mathbb{R}^d) \to L_2 (\mathbb{R}^d)} &\le C (1 + |t|)\varepsilon, \\
	\label{intro_cos_est}
	\| \cos(t \widehat{\mathcal{A}}_{\varepsilon}^{1/2}) - \cos(t (\widehat{\mathcal{A}}^{\mathrm{hom}})^{1/2}) \|_{H^2 (\mathbb{R}^d) \to L_2 (\mathbb{R}^d)} &\le  C (1 + |t|)\varepsilon.
\end{align}
In~\cite{M2021}, the result for the operator $\widehat{\mathcal{A}}_{\varepsilon}^{-1/2} \sin(t \widehat{\mathcal{A}}_{\varepsilon}^{1/2})$ was obtained:
\begin{equation}
	\label{intro_sin_est}
	\| \widehat{\mathcal{A}}_{\varepsilon}^{-1/2} \sin(t \widehat{\mathcal{A}}_{\varepsilon}^{1/2}) - (\widehat{\mathcal{A}}^{\mathrm{hom}})^{-1/2} \sin(t (\widehat{\mathcal{A}}^{\mathrm{hom}})^{1/2}) \|_{H^1 (\mathbb{R}^d) \to L_2 (\mathbb{R}^d)} \le C (1 + |t|)\varepsilon.
\end{equation}
Moreover, in~\cite{M2021}, an approximation of the operator $\widehat{\mathcal{A}}_{\varepsilon}^{-1/2} \sin(t \widehat{\mathcal{A}}_{\varepsilon}^{1/2})$ for a fixed $t$ in the $(H^2 \to H^1)$-norm with error of order $O (\varepsilon)$ (with a corrector taken into account) was obtained. Next, in~\cite{Su2017, DSu2018, DSu2020, D2021}, it was shown that these results are sharp with respect to the norm type as well as with respect to the dependence on $t$ (for large $t$). On the other hand, it was shown  that under some additional assumptions (e.g., if the matrix $g(\mathbf{x})$ has real entries) estimates~\eqref{intro_exp_est}--\eqref{intro_sin_est} can be improved:
\begin{alignat}{2}
	\notag
	\| e^{-i t \widehat{\mathcal{A}}_{\varepsilon}} - e^{-i t \widehat{\mathcal{A}}^{\mathrm{hom}}} &\|_{H^2 (\mathbb{R}^d) \to L_2 (\mathbb{R}^d)} &  &\le C (1 + |t|^{1/2})\varepsilon, 
	\\
	\label{intro_cos_improved_est}
	\| \cos(t \widehat{\mathcal{A}}_{\varepsilon}^{1/2}) - \cos(t (\widehat{\mathcal{A}}^{\mathrm{hom}})^{1/2}) &\|_{H^{3/2} (\mathbb{R}^d) \to L_2 (\mathbb{R}^d)} & &\le  C (1 + |t|^{1/2})\varepsilon,
	\\
	\label{intro_sin_improved_est}
	\| \widehat{\mathcal{A}}_{\varepsilon}^{-1/2} \sin(t \widehat{\mathcal{A}}_{\varepsilon}^{1/2}) - (\widehat{\mathcal{A}}^{\mathrm{hom}})^{-1/2} \sin(t (\widehat{\mathcal{A}}^{\mathrm{hom}})^{1/2}) &\|_{H^{1/2} (\mathbb{R}^d) \to L_2 (\mathbb{R}^d)} & &\le C (1 + |t|^{1/2})\varepsilon.
\end{alignat}

Note that in~\cite{BSu2003, Su2004,  Su2007, BSu2005, BSu2006, V, Su2010, BSu2008, Su2017, DSu2018, M2021, DSu2020, D2021}, a much broader class of operators than~\eqref{intro_hatA_eps} (including matrix DOs) was studied. In particular, operators of the form
\begin{equation}
	\label{intro_A_eps}
	\mathcal{A}_{\varepsilon} = - \operatorname{div} \check{g}^{\varepsilon}(\mathbf{x}) \nabla + \varepsilon^{-2}V^{\varepsilon}(\mathbf{x})
\end{equation}
were considered. Here $\check{g}(\mathbf{x})$ is a $\Gamma$-periodic positive definite and bounded $(d \times d)$-matrix-valued function with real entries, the potential $V(\mathbf{x})$ is a $\Gamma$-periodic real-valued function, $V \in L_p(\Omega)$ with a suitable $p$ (and it is assumed that $\inf \spec \mathcal{A}_1 = 0$). 
For operator~\eqref{intro_A_eps}, it is impossible to find an operator $\mathcal{A}^{\mathrm{hom}}$ with constant coefficients such that the corresponding operator-functions converge to the operator-functions of $\mathcal{A}^{\mathrm{hom}}$. However, some approximations can be found if we "border" operator-functions of $\widehat{\mathcal{A}}^{\mathrm{hom}}$ by appropriate rapidly oscillating factors. In particular, an analog of~\eqref{intro_BSu2003_est} is as follows:
\begin{equation*}
	\| (\mathcal{A}_\varepsilon + I)^{-1} - [\omega^\varepsilon](\widehat{\mathcal{A}}^{\mathrm{hom}}+ I)^{-1}[\omega^\varepsilon] \|_{L_2(\mathbb{R}^d) \to L_2(\mathbb{R}^d)} \le C \varepsilon,
\end{equation*} 
where $\omega (\mathbf{x})$ is a positive $\Gamma$-periodic solution of the equation
\begin{equation*}
	- \operatorname{div} \check{g}(\mathbf{x}) \nabla \omega (\mathbf{x}) + V(\mathbf{x}) \omega (\mathbf{x}) = 0
\end{equation*}
satisfying the normalization condition $\| \omega \|_{L_2(\Omega)}^2 = | \Omega |$, and $\widehat{\mathcal{A}}^{\mathrm{hom}}$ is the effective operator for operator~\eqref{intro_hatA_eps} with the matrix $g(\mathbf{x}) = \check{g} (\mathbf{x}) \omega^2 (\mathbf{x})$.

Let us explain the method using the example of operator~\eqref{intro_hatA_eps}. The scaling transformation reduces investigation of the behavior of the operator $(\widehat{\mathcal{A}}_\varepsilon + I)^{-1}$, $\varepsilon \to 0$, to studying the operator $(\widehat{\mathcal{A}} + \varepsilon^2 I)^{-1}$, where $\widehat{\mathcal{A}} = \widehat{\mathcal{A}}_1 = - \operatorname{div} g(\mathbf{x}) \nabla$. Next, by the Floquet--Bloch theory, the operator~$\widehat{\mathcal{A}}$ expands in the direct integral of the operators $\widehat{\mathcal{A}}(\mathbf{k})$ acting in the space $L_2 (\Omega)$. The operator $\widehat{\mathcal{A}}(\mathbf{k})$ is defined by the differential expression $-\operatorname{div}_\mathbf{k} g(\mathbf{x}) \nabla_\mathbf{k}$, where $\nabla_\mathbf{k} = \nabla + i \mathbf{k}$, $\operatorname{div}_\mathbf{k} = \operatorname{div} + i \langle \mathbf{k}, \cdot \rangle$, with periodic boundary conditions. The spectrum of the operator $\widehat{\mathcal{A}}(\mathbf{k})$ is discrete.
It turns out that the behavior of the resolvent $(\widehat{\mathcal{A}} + \varepsilon^2 I)^{-1}$ can be described in terms of the threshold characteristics of $\widehat{\mathcal{A}}$ at the edge of the spectrum, i.e., it is sufficient to know the spectral decomposition of $\widehat{\mathcal{A}}$ only near the lower edge of the spectrum. In particular, the effective matrix $g^0$ is a Hessian of the first band function $E_1(\mathbf{k})$ at the point $\mathbf{k} = 0$.

Finally, we mention the recent paper~\cite{LiSh2021}, where the authors investigated the problem of convergence rates for a solution of the initial-Dirichlet boundary value problem for a wave equation; analogs of estimates~\eqref{intro_cos_improved_est}, \eqref{intro_sin_improved_est} as well as results with the Dirichlet corrector were obtained.

\subsection{High-frequency homogenization.}
As stated above, only a small neighborhood of the bottom of the spectrum (i.e., waves with low frequencies) contributes to homogenization. However, we can consider problems of wave propagation when the frequency is proportional to $\varepsilon^{-1}$ or $\varepsilon^{-2}$ (the high-frequency mode). In this case, even the leading order of the asymptotics oscillates rapidly. These problems were studied in~\cite[Chapter~4]{BeLP} using WKB-ansatz.

Traditional methods of homogenization theory, related to asymptotic expansions in two scales, were applied to these problems in~\cite{CrKaPi2010, HaMiCr2016}. We also cite the paper~\cite{Ce_etal_2015}, where application of the results of~\cite{CrKaPi2010} to photonic crystals was considered. In~\cite{CrKaPi2010}, an asymptotic expansion for solutions of the equation
\begin{equation*}
	\operatorname{div} g^\varepsilon(\mathbf{x}) \nabla u_\varepsilon(\mathbf{x}) + \nu^2 \rho^\varepsilon(\mathbf{x}) u_\varepsilon(\mathbf{x}) = 0,
\end{equation*}
which are perturbations of the standing waves, was obtained (the functions $g(\mathbf{x})$, $\rho(\mathbf{x})$ were supposed to be sufficiently smooth and $\Gamma$-periodic). In~\cite{HaMiCr2016}, a similar problem for travelling waves was considered. In the paper~\cite{APi2004}, homogenization of the Cauchy problem for a nonstationary Schr\"{o}dinger equation with well-prepared initial data concentrating on a Bloch eigenfunction was studied using techniques of two-scale convergence and suitable oscillating test functions; a rigorous derivation of so-called effective mass theorems in solid state physics was obtained. We also mention the papers~\cite{KuRa2012,KhKuRa2017}, where asymptotics of Green’s function for different values of the spectral parameter has been studied.

Now, let us discuss error estimates for high-frequency homogenization. This topic has been studied in~\cite{B03,SuKh2009,AkhAkSlSu2021} in the one-dimensional case ($d=1$) and in~\cite{BSu2004, SuKh2011} in the case of arbitrary dimension $d$. It is well-known that the spectrum of $\mathcal{A}$ has a band structure and may have gaps. For the sake of simplicity, we consider the case where $d=1$ and $\Gamma = \mathbb{Z}$;  in this case we shall use the notation $A_\varepsilon$ for operator~\eqref{intro_A_eps}. Let $\sigma > 0$ be a (non-degenerate) left edge of a band with an odd number ($\ge 3$) in the spectrum of the operator $A = A_1$. Then for $A_\varepsilon$, this edge "moves" to the point $\varepsilon^{-2} \sigma$ (to the high-frequency (high-energy) region). Instead of~\eqref{intro_elliptic_eq}, we consider the
equation
\begin{equation}
	\label{intro_int_edge_elliptic_eq_est}
	- \frac{d}{dx} g^\varepsilon(x) \frac{d}{dx} u_\varepsilon(x) - (\varepsilon^{-2} \sigma - \varkappa^2) u_\varepsilon(x) = f(x),
\end{equation} 
where $f \in L_2(\mathbb{R})$. It is supposed that $\varkappa > 0$ is such that the point $\varepsilon^{-2} \sigma - \varkappa^2$ belongs to the gap in the spectrum of the operator $A_\varepsilon$. Similarly to~\eqref{intro_BSu2003_est}, the question is reduced to studying the operator $(A_\varepsilon - (\varepsilon^{-2} \sigma - \varkappa^2)I)^{-1}$. In~\cite{B03}, the following result was proved:
\begin{equation}
	\label{intro_int_gap_est}
	\| (A_\varepsilon - (\varepsilon^{-2} \sigma - \varkappa^2)I)^{-1} - [\varphi_\sigma^\varepsilon](A_\sigma^{\mathrm{hom}} + \varkappa^2 I)^{-1}[\varphi_\sigma^\varepsilon] \|_{L_2(\mathbb{R}) \to L_2(\mathbb{R})} \le C \varepsilon.
\end{equation} 
Here $A_\sigma^{\mathrm{hom}} = -b_\sigma \frac{d^2}{dx^2}$ is the corresponding effective operator, $b_\sigma > 0$ is the coefficient in the asymptotics of the band function $E(k)$ corresponding to the
band for which $\sigma$ is the left edge: $E(k) \sim \sigma + b_\sigma k^2$, $k \sim 0$; and $\varphi_\sigma$ is a real-valued periodic solution of the equation $A \varphi_\sigma = \sigma \varphi_\sigma$, normalized in $L_2(0,1)$. Consequently, the possibility of homogenization for equation~\eqref{intro_int_edge_elliptic_eq_est} is a threshold effect near the edge of an internal gap.

Estimate~\eqref{intro_int_gap_est} was obtained in~\cite{B03} in the case where $V(x) = 0$. In~\cite{BSu2004}, an analog of estimate~\eqref{intro_int_gap_est} was proved for operators~\eqref{intro_A_eps} in arbitrary dimension $d \ge 1$. More accurate approximations with correctors were obtained in~\cite{SuKh2009, SuKh2011}. 

Parabolic equations in the one-dimensional case were studied in~\cite{AkhAkSlSu2021}. It was proved that
\begin{equation*}
	\| e^{-t A_{\varepsilon}} \mathcal{E}_{A_\varepsilon}[\varepsilon^{-2} \sigma, \infty) - e^{-t \sigma/\varepsilon^2} [\varphi_\sigma^\varepsilon]e^{-t A_\sigma^{\mathrm{hom}}}[\varphi_\sigma^\varepsilon] \|_{L_2(\mathbb{R}) \to L_2(\mathbb{R})} \le C e^{-t \sigma/\varepsilon^2} \varepsilon (t + \varepsilon^2)^{-1/2}, \quad t > 0,
\end{equation*}
and a more accurate approximation with a corrector was found. Here $\mathcal{E}_{A_\varepsilon}[\varepsilon^{-2} \sigma, \infty)$ is the spectral projection of the operator $A_\varepsilon$ corresponding to the
interval $[\varepsilon^{-2} \sigma, \infty)$.

\subsection{Main results.}
In the present paper, we study error estimates for high-frequency homogeniza\-tion of nonstationary Schr\"{o}dinger equations and hyperbolic equations in the one-dimensional case ($d=1$). Main results of the paper are formulated in Section~\ref{main_results_section}. (Note that in Section~\ref{main_results_section}, it is convenient for us to use  slightly different notations.) In the introduction, we consider again only the case where  $\sigma$ is a (non-degenerate) left band edge with an odd number $s$ in the spectrum of the operator $A$. 

Let $f, g \in L_2(\mathbb{R})$. Consider the Cauchy problems
\begin{equation}
	\label{intro_Cauchy_problem_Schrod_hyperb}
	\left\{
	\begin{aligned}
		&i \frac{\partial}{\partial t} u_\varepsilon (x,t) = (A_\varepsilon u_\varepsilon)(x,t) ,\\
		&u_\varepsilon (x,0) = (\Upsilon_\varepsilon f)(x),
	\end{aligned}
	\right.
	\qquad 
		\left\{
		\begin{aligned}
			&\frac{\partial^2}{\partial t^2}  v_\varepsilon (x,t) = - (A_\varepsilon v_\varepsilon)(x,t) + \varepsilon^{-2} \sigma v_\varepsilon(x,t),\\
			&v_\varepsilon (x,0) = (\Upsilon_\varepsilon f)(x), \; (\partial_t v_\varepsilon) (x,0) = (\Upsilon_\varepsilon g)(x),
		\end{aligned}
		\right.
\end{equation} 
where
\begin{equation*}
	(\Upsilon_\varepsilon f)(x) \coloneqq (2 \pi)^{-1/2} \int_\mathbb{R} (\Phi f) (k) \sum_{j=s}^{\infty} e^{ikx} \varphi_j(x/\varepsilon, \varepsilon k) \chi_{\widetilde{\Omega}_{j-s+1}} (\varepsilon k) \, dk.
\end{equation*}
Here $\{e^{ikx} \varphi_j(x, k)\}_{j=s}^{\infty}$ are the Bloch waves corresponding to the bands with the numbers $j \ge s$;
\begin{equation}
	\label{Brillouin_zones}
	\widetilde{\Omega}_j = (-j \pi, -(j-1)\pi] \cup ((j-1)\pi, j \pi ], \qquad j \in \mathbb{N},
\end{equation}
are the Brillouin zones. The initial data of problems~\eqref{intro_Cauchy_problem_Schrod_hyperb} are superpositions of the Bloch waves with the amplitudes, which are equal to the Fourier images $(\Phi f) (k)$, $(\Phi g) (k)$ of the functions $f(x)$, $g(x)$, and belong to the subspace $\mathcal{E}_{A_\varepsilon}[\varepsilon^{-2} \sigma, \infty) L_2(\mathbb{R})$. Main results of the paper are the following estimates:
\begin{align}
	\label{intro_Schrod_est}
	&\| u_\varepsilon(\cdot, t) - e^{-i t \varepsilon^{-2} \sigma} \varphi_\sigma^{\varepsilon}  u_0(\cdot, t)\|_{L_2(\mathbb{R})} \le C (1+|t|^{1/2}) \varepsilon \|f\|_{H^2(\mathbb{R})}, \qquad  f \in H^2(\mathbb{R}),
	\\
	\label{intro_hyperb_est}
	&\begin{multlined}[c][0.9\textwidth]
		\| v_\varepsilon(\cdot, t) - \varphi_\sigma^{\varepsilon}  v_0(\cdot, t)\|_{L_2(\mathbb{R})} \le C (1+|t|^{1/2}) \varepsilon (\|f\|_{H^{3/2}(\mathbb{R})} + \|g\|_{H^{1/2}(\mathbb{R})}), \\ f \in H^{3/2}(\mathbb{R}), \; g \in H^{1/2}(\mathbb{R}).
	\end{multlined}	
\end{align}
Here $u_0$ and $v_0$ are the solutions of the effective problems
\begin{equation*}
	\left\{
	\begin{aligned}
		&i \frac{\partial}{\partial t} u_0 (x,t) = (A^\textrm{hom}_\sigma u_0)(x,t) ,\\
		&u_0 (x,0) = f (x),
	\end{aligned}
	\right.
	\qquad
	\left\{
	\begin{aligned}
		&\frac{\partial^2}{\partial t^2} v_0 (x,t) = - (A^\textrm{hom}_\sigma v_0)(x,t) ,\\
		&v_0 (x,0) = f (x), \; (\partial_t v_0) (x,0) = g (x),
	\end{aligned}
	\right.
\end{equation*}
and $A^\textrm{hom}_\sigma$, $\varphi_\sigma$ are the same as in~\eqref{intro_int_gap_est}.

Note that estimates~\eqref{intro_Schrod_est}, \eqref{intro_hyperb_est} can be formulated in operator terms, but we postpone the discussion until Section~\ref{main_results_section}.

\subsection{Notation.}
Let $\mathfrak{H}_1$ and $\mathfrak{H}_{2}$ be complex separable Hilbert spaces. If $A \colon \mathfrak{H}_1 \to \mathfrak{H}_2$ is a closed linear operator, then $\Dom A$ stands for its domain, the adjoint operator is denoted by $A^*$. 
The symbol $\| \cdot \|_{\mathfrak{H}_1 \to \mathfrak{H}_2}$ denotes the norm of a linear bounded operator from $\mathfrak{H}_1$ to $\mathfrak{H}_2$. Next, if $A$ is a selfadjoint operator in some Hilbert space, then we use the notation $\spec A$ for the spectrum of $A$, $\mathcal{E}_A(\delta)$ stands for the spectral projection of the operator $A$ corresponding to the Borel set $\delta \subset \mathbb{R}$.

The standard $L_p$ classes of functions on an interval $(a,b) \subset \mathbb{R}$ are denoted by $L_p (a,b)$, $1 \le p \le \infty$. If $f$ is a measurable function, then $[f]$ or $[f(x)]$ denote the operator of multiplication by the function $f$ in the space $L_2$. Next, $H^s (\mathbb{R})$ is the Sobolev class of order $s \in \mathbb{R}$ and integrability index $2$; and $\widetilde{H}^1(0,1)$ is the subspace formed by the functions from $H^1(0,1)$ whose $1$-periodic extensions belong to $H^1_{\mathrm{loc}}(\mathbb{R})$.

If $F(x)$ is a $1$-periodic function, then we put  $F^\varepsilon (x) \coloneqq F(\varepsilon^{-1} x)$. By $\Phi \coloneqq \Phi_{x \to k}$ we denote the Fourier transform on $\mathbb{R}$ defined on the Schwartz class by the formula
\begin{equation*}
	(\Phi v) (k) = (2 \pi)^{-1/2} \int_{\mathbb{R}} e^{-i k x} v(x) \, d x, \qquad v \in \mathcal{S}(\mathbb{R}),
\end{equation*}
and extended by continuity up to the unitary mapping $\Phi \colon L_2(\mathbb{R}) \to L_2(\mathbb{R})$. For the characteristic function of a set $\delta \subset \mathbb{R}$, we use the notation $\chi_\delta$. The range of values of an arbitrary function $f$ is denoted by $R(f)$.

\subsection{Acknowledgments.} The author is grateful to T.~A.~Suslina for helpful discussions and attention to the work.

\section{The operator $A$}
Let $A$ be a self-adjoint operator in $L_2(\mathbb{R})$ generated by the differential expression
\begin{equation}
	\label{A_with_V}
	A = - \frac{d}{d x} \check{g} (x) \frac{d}{d x} + V(x),
\end{equation}
where 
\begin{equation}
	\label{g_check_cond}
	\left. 
	\begin{aligned}
		\check{g}~\text{is a real-valued measurable function,}&\\
		0 < \alpha_0 \le \check{g} (x) \le \alpha_1 < \infty, \; \check{g}(x+1) = \check{g}(x), \; x \in \mathbb{R},&
	\end{aligned}	
	\quad\right\rbrace 
\end{equation}
and where $V(x)$ is a real-valued potential such that $V \in L_1(0,1)$, $V(x+1) = V(x)$, $x \in \mathbb{R}$.
The precise definition of the operator $A$  is given in terms of the semi-bounded closed quadratic form
\begin{equation}
	\label{a_form_1}
	\mathfrak{a}[u,u] = \int_{\mathbb{R}} (\check{g} (x) |u' (x)|^2 + V(x) |u(x)|^2) \, d x, \qquad u \in H^1(\mathbb{R}).
\end{equation}
Adding an appropriate constant to $V$, we assume that $\inf \spec A = 0$. Under this assumption the operator $A$ admits a convenient factorization (see, e.g., \cite{KiSi1987}, \cite[Chapter~6, Section~1.1]{BSu2003}). To describe this factorization, we consider the equation
\begin{equation*}
	- (\check{g}(x) \omega' (x))' + V(x) \omega (x) = 0
\end{equation*}
(which is understood in the weak sense). There exists a $1$-periodic solution $\omega \in \widetilde{H}^1 (0,1)$ of this equation defined up to a constant factor. This factor can be fixed  so that $\omega (x) > 0$ and
\begin{equation}
	\label{omega_cond1}
	\| \omega \|_{L_2(0,1)} = 1.
\end{equation}
It turns out that the solution $\omega (x)$ is positive definite and bounded:
\begin{equation}
	\label{omega_cond2}
	0 < \beta_0 \le \omega (x) \le \beta_1 < \infty.
\end{equation}
The function $\omega$ is Lipschitz and is a multiplier in $H^1 (\mathbb{R})$ and in $\widetilde{H}^1 (0,1)$. The substitution  $u = \omega \phi$ transforms form~\eqref{a_form_1} to the form
\begin{equation}
	\label{a_form_2}
	\mathfrak{a}[u,u] = \int_{\mathbb{R}^d} \omega^2 (x) \check{g}(x) |\phi'(x)|^2 \, dx, \qquad u = \omega \phi, \quad \phi \in H^1(\mathbb{R}).
\end{equation}
This means that operator~\eqref{A_with_V} admits the following factorization:
\begin{equation}
	\label{A}
	A = - \omega(x)^{-1} \frac{d}{d x} g(x) \frac{d}{d x} \omega(x)^{-1}, \qquad g = \omega^2 \check{g}.
\end{equation}
We take representation~\eqref{A} of the operator $A$ as the initial definition, i.e., we assume that $A$ is the operator generated by form~\eqref{a_form_2}, where $\check{g}$ and $\omega$ are $1$-periodic functions satisfying~\eqref{g_check_cond}, \eqref{omega_cond1}, \eqref{omega_cond2}. We can return to representation~\eqref{A_with_V} putting $V(x) = (\check{g}(x) \omega'(x))'/\omega(x)$. However, the potential $V(x)$ may be highly singular.

\section{Spectral decomposition of operator~\eqref{A}}
We need to describe the spectral properties of operator~\eqref{A}. For this, let us introduce the objects associated with the spectral resolution of operator~\eqref{A}. We follow the papers~\cite{B03,SuKh2009,AkhAkSlSu2021}, see also~\cite[Chapter~2, \S\S\,1,2]{BSu2003}.  
In $L_2(0,1)$, consider the family of quadratic forms 
\begin{equation}
	\label{a(k)_form}
	\mathfrak{a}(k)[u,u] = \int_0^1 g (x) |\phi' + i k \phi|^2 \, d x, \qquad \phi = \omega^{-1} u \in \widetilde{H}^1(0,1), \quad k \in \mathbb{R}.
\end{equation}
The operator generated by form~\eqref{a(k)_form} is denoted by $A (k)$. The parameter $k \in \mathbb{R}$ is called the \emph{quasimomentum}.  Let $E_l (k)$, $l \in \mathbb{N}$, be consecutive (counted with multiplicities)
eigenvalues of the operator $A (k)$, and let $\varphi_l (\cdot, k)$, $l \in \mathbb{N}$, be the corresponding normalized eigenfunctions. The functions $E_l (k)$ are called \emph{band functions}; they are $(2 \pi)$-periodic. Next, $\varphi_l (x+1, k) = \varphi_l (x, k)$, and the functions $e^{ikx} \varphi_l (x, k)$ can be chosen to be $(2 \pi)$-periodic in $k$.

Denote by $\widetilde{\Omega} = \widetilde{\Omega}_1 = (-\pi, \pi]$ the central (first) Brillouin zone. Since the functions  $E_l (k)$ and $e^{ikx} \varphi_l (x, k)$ are periodic, it suffices to consider $k \in \widetilde{\Omega}$ only. However, it will be sometimes convenient to assume that $k \in \mathbb{R}$. 

Consider the function $E_s (k)$ for some $s \in \mathbb{N}$. The following facts are well known.
\begin{enumerate}[label=\arabic*$^{\circ}.$, ref=\arabic*$^{\circ}$, leftmargin=3\parindent]
	\setlength\itemsep{-0.1em}
	\item The function $E_s$ is Lipschitz and even.
	\item The mapping $k \mapsto E_s(k)$, $k \in \widetilde{\Omega}$, covers the band $R (E_s)$ twice.
	\item The function $E_s$ is piecewise real-analytic, its smoothness may be lost at the points where $E_{s+1} (k) = E_s(k)$ or $E_{s-1} (k) = E_s(k)$.
	\item The equality $E_{s+1} (k) = E_s(k)$, $k \in \widetilde{\Omega}$, is possible only if $k =  \pi$ ($s$ is an odd number).
	\item The equality $E_{s-1} (k) = E_s(k)$, $k \in \widetilde{\Omega}$, is possible only if $k = 0$ ($s$ is an odd number).	
	\item For $0 \le k \le \pi$ the function $E_s(k)$ is strictly monotone.
	\item If the number $s$ is odd, then the function $E_s (k)$, $k \in \widetilde{\Omega}$, has its minimum value at the point $k=0$ and its maximum value at the point $k = \pi$.
	\item If the number $s$ is even, then the function $E_s (k)$, $k \in \widetilde{\Omega}$, has its maximum value at the point $k=0$ and its minimum value at the point $k = \pi$.
\end{enumerate}

We need estimates for the band functions $E_l(k)$, $l \in \mathbb{N}$. Consider the form
\begin{equation*}
	\mathfrak{a}^\circ (k)[\phi,\phi] = \int_0^1 |\phi' + i k \phi|^2 \, d x, \qquad \phi \in \widetilde{H}^1(0,1), \quad k \in \mathbb{R}. 
\end{equation*}
Let $E^\circ_l (k)$, $l \in \mathbb{N}$, be consecutive (counted with multiplicities) eigenvalues of the corresponding operator. They are $(2 \pi)$-periodic and are equal to 
\begin{gather*}
	E^\circ_1(k)   =  k^2, \qquad k \in \widetilde{\Omega},
	\\
	E^\circ_{2j}(k) = (2 \pi j - |k|)^2, \qquad E^\circ_{2j+1}(k) = (2 \pi j + |k|)^2,  \qquad k \in \widetilde{\Omega}, \quad j \in \mathbb{N},
\end{gather*}  
From periodicity of the functions $\{E^\circ_l (k)\}_{l \in \mathbb{N}}$ it follows that
\begin{equation}
	\label{E0_l_k^2}
	E^\circ_l (k) \chi_{\widetilde{\Omega}_{l}}(k) = k^2 \chi_{\widetilde{\Omega}_{l}}(k), \qquad k \in \mathbb{R}, \quad l \in \mathbb{N}.
\end{equation}
Here $\widetilde{\Omega}_l$ are the Brillouin zones~\eqref{Brillouin_zones}. Next, using the Fourier series, one can show that
\begin{equation*}
	 \alpha_0 \beta_0^2 \mathfrak{a}^\circ (k)[\phi,\phi]\le \mathfrak{a} (k) [u,u] \le  \alpha_1 \beta_1^2 \mathfrak{a}^\circ (k)[\phi,\phi], \qquad \phi = \omega^{-1} u \in \widetilde{H}^1(0,1), \quad k \in \mathbb{R},
\end{equation*}
whence
\begin{equation*}
	\alpha_0 \beta_0^2 \beta_1^{-2} \frac{\mathfrak{a}^\circ (k)[\phi,\phi]}{\|\phi\|_{L_2(0,1)}^2} \le \frac{\mathfrak{a} (k) [u,u]}{\|u\|_{L_2(0,1)}^2}  \le  \alpha_1 \beta_0^{-2} \beta_1^2 \frac{\mathfrak{a}^\circ (k)[\phi,\phi]}{\|\phi\|_{L_2(0,1)}^2}, \qquad \phi = \omega^{-1} u \in \widetilde{H}^1(0,1), \quad k \in \mathbb{R}.
\end{equation*}
Using the variational principle, periodicity of the functions $\{E_l (k)\}_{l\in \mathbb{N}}$, and~\eqref{E0_l_k^2}, we get
\begin{equation}
	\label{E_l_estimates}
	\alpha_0 \beta_0^2 \beta_1^{-2} k^2 \chi_{\widetilde{\Omega}_l}(k) \le E_l (k) \chi_{\widetilde{\Omega}_l}(k) \le \alpha_1 \beta_0^{-2} \beta_1^2 k^2 \chi_{\widetilde{\Omega}_l}(k), \qquad k \in \mathbb{R}, \quad l \in \mathbb{N}.
\end{equation}

Now we introduce integral operators
\begin{equation}
	\label{Psi_l}
	\Psi_l \colon L_2(\mathbb{R}) \to L_2(\mathbb{S}^1), \qquad l \in \mathbb{N},
\end{equation}
defined on the Schwartz class by the following formula:
\begin{equation}
	\label{Psi_l_integral}
	(\Psi_l v) (k) = (2 \pi)^{-1/2} \int_{\mathbb{R}} e^{-i k x} \overline{\varphi_l (x, k)} v(x) \, d x, \qquad v \in \mathcal{S}(\mathbb{R}), \quad l \in \mathbb{N}.
\end{equation}
Points of the circle $k \in \mathbb{S}^1$ can be realized, for example, as points from $\widetilde{\Omega}$. However, other realizations will be also convenient for us. Operators~\eqref{Psi_l_integral} extend by continuity up to partial isometries "onto". The operators $\Psi_l^* \Psi_l$ are orthogonal projections of $L_2(\mathbb{R})$ onto $\mathcal{E}_{A}( R(E_l)) L_2(\mathbb{R})$, and $\sum_{l \in \mathbb{N}} \Psi_l^* \Psi_l = I$.
The following representation (the Floquet--Bloch decomposition) is true:
\begin{equation}
	\label{decomp}
	A = \sum_{l=1}^{\infty} \Psi_l^* [E_l] \Psi_l.
\end{equation}

From~\eqref{decomp} it follows that the spectrum of the operator $A$ is the union of intervals (bands), which are the ranges of the functions $E_l (k)$: 
\begin{equation*}
	\sigma(A) = \bigcup_{l=1}^\infty R (E_l) = [E_1(0), E_1(\pi)] \cup [E_2(\pi), E_2(0)] \cup [E_3(0), E_3(\pi)] \cup \ldots.
\end{equation*}
In the one-dimensional case the \emph{spectral bands} cannot overlap. The intervals
\begin{equation*}
	 (- \infty, E_1(0)),  (E_1(\pi), E_2(\pi)), (E_2(0), E_3(0)), \ldots
\end{equation*}
are called \emph{spectral gaps}. Note that some bands may touch, i.e., they may intersect at the boundary points. This means that some gaps may be empty.

\section{A gap}
Let us fix some $s \in \mathbb{N}$. Here and throughout the paper we agree to denote $E(k) \coloneqq E_s(k)$, $\varphi(x, k) \coloneqq \varphi_s(x, k)$, $\Psi \coloneqq\Psi_s$. Put
\begin{equation*}
		\widetilde{\mathcal{H}}^1(0,1) \coloneqq \{f \colon \omega^{-1} f \in \widetilde{H}^1(0,1)\}, \qquad \|f\|_{\widetilde{\mathcal{H}}^1(0,1)} \coloneqq \| \omega^{-1} f \|_{H^1(0,1)}.
\end{equation*}

We suppose that at least one of the following four conditions is fulfilled.
\begin{condition}
	\label{cond1}
	Let $s$ be an odd number. There is the gap $(E_{s-1}(0), E(0)) \ne \varnothing$ in the spectrum of the operator $A$.
\end{condition}
\begin{condition}
	\label{cond2}
	Let $s$ be an even number. There is the gap $(E(0), E_{s+1}(0)) \ne \varnothing$ in the spectrum of the operator $A$.
\end{condition}
Under the assumptions of Condition~\ref{cond1} or Condition~\ref{cond2}, we have
\begin{alignat}{2}
	\label{E_powerseries_1}
	&E (k) = \sigma_0 \pm  b_0 k^2 + k^4 \gamma_0 (k), & \qquad &|k| \le \pi, \;  b_0 > 0,\\
	\label{phi_powerseries_1}
	&\varphi (x, k) = \varphi_0 (x) + k \theta_0 (x,k), & &|k| < \pi.
\end{alignat}
In equality~\eqref{E_powerseries_1}, the sign "$+$" corresponds to the case where Condition~\ref{cond1} is fulfilled, "$-$" to the case where Condition~\ref{cond2} is fulfilled. The function $\gamma_0(k)$ is continuous, and for $|k| < \pi$ this function is real-analytic; $\sigma_0 \coloneqq E(0)$, $\varphi_0 (x) \coloneqq \varphi(x, 0)$; $\varphi (\cdot, k)$ and $\theta_0 (\cdot,k)$ are real-analytic functions of $k \in (-\pi,\pi)$ with values in $\widetilde{\mathcal{H}}^1(0,1)$.

\begin{condition}
	\label{cond3}
	Let $s$ be an odd number. There is the gap $(E(\pi), E_{s+1}(\pi)) \ne \varnothing$ in the spectrum of the operator $A$.
\end{condition}
\begin{condition}
	\label{cond4}
	Let $s$ be an even number. There is the gap $(E_{s-1}(\pi), E(\pi)) \ne \varnothing$ in the spectrum of the operator $A$.
\end{condition}
Under the assumptions of Condition~\ref{cond3} or Condition~\ref{cond4}, we have
\begin{alignat}{2}
	\label{E_powerseries_2}
	&E (k) = \sigma_{\pi} \pm  b_{\pi} (k - \pi)^2 + (k - \pi)^4 \gamma_{\pi} (k), & \qquad &0 \le k \le 2\pi, \; b_{\pi} > 0,\\
	\label{phi_powerseries_2}
	&\varphi (x, k) = \varphi_{\pi}(x) + (k - \pi) \theta_{\pi}(x,k), & &0 < k < 2\pi.
\end{alignat}
In equality~\eqref{E_powerseries_2}, the sign "$+$" corresponds to the case where Condition~\ref{cond4} is fulfilled, "$-$" to the case where Condition~\ref{cond3} is fulfilled. The function $\gamma_{\pi} (k)$ is continuous, and for $0 < k < 2\pi$ this function is real-analytic; $\varphi_{\pi}(x) \coloneqq \varphi(x, \pi)$; $\varphi (\cdot, k)$ and $\theta_{\pi} (\cdot,k)$ are real-analytic functions of $k \in (0, 2\pi)$ with values in $\widetilde{\mathcal{H}}^1(0,1)$.
\begin{remark}
	In order to include the case of the semi-infinite gap $(-\infty, 0)$ under the assumptions of  Condition~\textup{\ref{cond1}}, we formally put $E_0(0) = - \infty$.
\end{remark}

\begin{remark}
	The functions $\varphi_{0}$ and $\varphi_{\pi}$ in~\eqref{phi_powerseries_1}, \eqref{phi_powerseries_2} belong to $\widetilde{\mathcal{H}}^1(0,1)$ and therefore are bounded. We suppose that they are $1$-periodically extended to  $\mathbb{R}$. 
\end{remark}

\begin{remark}
	The coefficients $b_0$, $b_{\pi}$ in~\eqref{E_powerseries_1} and~\eqref{E_powerseries_2} can be expressed in terms of solutions of some auxiliary boundary value problems on the interval of periodicity $(0,1)$. See, e.g.,~\textup{\cite[Remarks~2.2, 2.4]{AkhAkSlSu2021}}.
\end{remark}

\section{Auxiliary statements}
Put $k_0 \coloneqq 0$ if Condition~\ref{cond1} or~\ref{cond2} is fulfilled, and put  $k_0 \coloneqq \pi$ if Condition~\ref{cond3} or~\ref{cond4} is fulfilled. Here and throughout the paper we assume that
\begin{equation}
	\label{gamma_ne_0}
	\gamma_{k_0}(k_0) \ne 0.
\end{equation}

Suppose that Condition~\ref{cond1} or Condition~\ref{cond4} is satisfied. Consider the expression $(E(k) - \sigma_{k_0})^{1/2}$. Using the Taylor formula for the function $\sqrt{1 + x}$, we have
\begin{equation}
	\label{sqrt_E_s_series_k=0}
	(E(k) - \sigma_{k_0})^{1/2}  = b_{k_0}^{1/2} |k-k_0| + |k-k_0|^3 \tilde{\gamma}_{k_0}(k),
\end{equation}
where $\tilde{\gamma}_{k_0}(k) = \frac{1}{2} b_{k_0}^{-1/2} \gamma_{k_0}(k) (1 + O((k-k_0)^2))$, $k \sim k_0$.
Similarly, if Condition~\ref{cond2} or Condition~\ref{cond3} is satisfied, we have
\begin{equation*}
	(\sigma_{k_0} - E(k))^{1/2}  = b_{k_0}^{1/2} |k-k_0|  - |k-k_0|^3 \tilde{\gamma}_{k_0}(k).
\end{equation*}

Let us fix $0 < \kappa < \pi$ so that 
\begin{gather}
	\label{delta_fix}
	\frac{1}{2} |\gamma_{k_0} (k_0)| \le  |\gamma_{k_0} (k)| \le \frac{3}{2} |\gamma_{k_0} (k_0)|,
	\\
	\label{delta_fix2}
	\frac{1}{2} |\tilde{\gamma}_{k_0} (k_0)| \le  |\tilde{\gamma}_{k_0} (k)| \le \frac{3}{2} |\tilde{\gamma}_{k_0} (k_0)|,
	\\
	\label{delta_fix3}
	(k-k_0)^2 |\tilde{\gamma}_{k_0} (k)| \le \frac{1}{2} b_{k_0}^{1/2},
\end{gather}
where $|k-k_0| \le \kappa$, and denote $\mathfrak{K} \coloneqq \{k \colon |k| \le \kappa\}$.

\begin{lemma}
	\label{lemma1}
	Suppose that one of Conditions~\textup{\ref{cond1}\,--\,\ref{cond4}} is fulfilled. Let $q \ge 0$, $r \ge -1$. Then for $0 < \varepsilon \le 1$ we have
	\begin{align}
		\label{lemma1_est_1}
		\bigl\| (\Psi^* - [\varphi_{k_0}]\Phi^*) [\varepsilon^{q} (|k-k_0|^2 + \varepsilon^2)^{-q/2} \chi_\mathfrak{K}(k-k_0)] \bigr\|_{L_2(\mathbb{R}) \to L_2(\mathbb{R})} &\le C \varepsilon^{\min(1,q)},
		\\
		\label{lemma1_est_2}
		\bigl\| (\Psi^* - [\varphi_{k_0}]\Phi^*) [\varepsilon^{r} (|k-k_0|^2 + \varepsilon^2)^{-r/2} |k-k_0|^{-1} \chi_\mathfrak{K}(k-k_0)] \bigr\|_{L_2(\mathbb{R}) \to L_2(\mathbb{R})} &\le C \varepsilon^{\min(0,r)}.
	\end{align}
The constant $C$ from~\eqref{lemma1_est_1}  depends on $q$, $\kappa$, $\|\theta_{k_0}\|_{M_\kappa(k_0)}$\textup{;} the constant $C$ from~\eqref{lemma1_est_2} depends on $r$, $\kappa$, $\|\theta_{k_0}\|_{M_\kappa(k_0)}$\textup{;} the quantity $\|\theta_{k_0}\|_{M_\kappa(k_0)}$ is defined below in~\eqref{mult_norm}.
\end{lemma}
\begin{proof}
	Let us prove~\eqref{lemma1_est_1}. Consider the adjoint operator
	\begin{equation*}
		[\varepsilon^q (|k-k_0|^2 + \varepsilon^2)^{-q/2} \chi_\mathfrak{K}(k-k_0)] (\Psi - \Phi[\overline{\varphi_{k_0}}]).
	\end{equation*}
	By~\eqref{phi_powerseries_1}, \eqref{phi_powerseries_2}, we conclude that the operator $[\chi_\mathfrak{K}(k-k_0)] (\Psi - \Phi[\overline{\varphi_{k_0}}])$ is the integral operator with the kernel
	\begin{equation}
		\label{lemma1_proof_f1}
		(2 \pi)^{-1/2} \chi_\mathfrak{K}(k-k_0) e^{-i x k} (\overline{\varphi (x, k)} - \overline{\varphi_{k_0}(x)}) = (2 \pi)^{-1/2} \chi_\mathfrak{K}(k-k_0) e^{-i x k} (k-k_0) \overline{\theta_{k_0}(x, k)}.
	\end{equation}
	Kernel~\eqref{lemma1_proof_f1} differs from the kernel of the Fourier operator by the factor $(k-k_0) \chi_\mathfrak{K}(k-k_0) \overline{\theta_{k_0}(x, k)}$. The function $\overline{\theta_{k_0}(x, k)}$ is a multiplier on the set of kernels of bounded integral operators that take $L_2(\mathbb{R})$ to $L_2(k_0-\kappa,k_0+\kappa)$ (see~\cite[\S\,2, Sec.~2]{B03}), and its norm in the class of multipliers equals
	\begin{equation}
		\label{mult_norm}
		\|\theta_{k_0}\|_{M_\kappa(k_0)} \coloneqq \underset{x \in \mathbb{R}}{\operatorname{ess-sup}\,} \| \theta_{k_0} (x, \cdot) \|_{C^1[k_0-\kappa,k_0+\kappa]} < \infty.
	\end{equation}
	Together with the inequalities
	\begin{alignat*}{2}
		&\varepsilon^{q} |k-k_0| \bigl((k-k_0)^2 + \varepsilon^2\bigr)^{-q/2} \chi_\mathfrak{K}(k-k_0)  \le \kappa^{1-q} \varepsilon^{q}, & \qquad &\text{for} \quad 0 \le q \le 1,
		\\
		&\varepsilon^{q} |k-k_0| \bigl((k-k_0)^2 + \varepsilon^2\bigr)^{-q/2} \chi_\mathfrak{K}(k-k_0)  \le \varepsilon, & \qquad &\text{for} \quad q > 1,		
	\end{alignat*}
  this yields~\eqref{lemma1_est_1}.
  
  Estimate~\eqref{lemma1_est_2} is proved in a similar way taking into account the inequalities
  \begin{alignat*}{2}
  	&|k-k_0| \cdot |k-k_0|^{-1} \varepsilon^{r} ((k-k_0)^2 + \varepsilon^2)^{-r/2} \chi_\mathfrak{K}(k-k_0)  \le (\kappa^2+1)^{-r/2} \varepsilon^{r}, & \qquad &\text{for} \quad -1 \le r < 0,\\
  	&|k-k_0| \cdot |k-k_0|^{-1} \varepsilon^{r} ((k-k_0)^2 + \varepsilon^2)^{-r/2} \chi_\mathfrak{K}(k-k_0)  \le 1, & \qquad &\text{for} \quad r \ge 0.  	
  \end{alignat*}
\end{proof}

\begin{lemma}
	\label{lemma2}
	Suppose that one of Conditions~\textup{\ref{cond1}\,--\,\ref{cond4}} is fulfilled, and condition~\eqref{gamma_ne_0} is also fulfilled. For $t \ne 0$ we have
	\begin{align}
		\label{lemma2_est_1}
		&\begin{multlined}[c][0.9\textwidth]
			\left\| \left[ \frac{\varepsilon^q}{((k-k_0)^2 + \varepsilon^2)^{q/2}} \sin\left( \frac{1}{2} t \varepsilon^{-2} (k-k_0)^4 \gamma_{k_0} (k) \right) \chi_\mathfrak{K}(k-k_0) \right] \right\|_{L_2(\mathbb{R}) \to L_2(\mathbb{R})}  
			\\
			\asymp \begin{dcases}
				\frac{\varepsilon^{q/2} |t|^{q/4}}{(|\gamma_{k_0} (k_0)|^{-1/2} + \varepsilon |t|^{1/2})^{q/2}}, & \text{for} \quad 0 \le q \le 4 \quad (\text{and} \quad 0 < \varepsilon |t|^{-1/2} \le \mathfrak{e}),
				\\
				\vphantom{\frac{0}{0}} |\gamma_{k_0} (k_0)| \, \varepsilon^2 |t|, & \text{for} \quad q > 4,
			\end{dcases}
		\end{multlined}
		\\
		\label{lemma2_est_2}
		&\begin{multlined}[c][0.9\textwidth]
			\left\| \left[ \frac{\varepsilon^{q}}{((k-k_0)^2 + \varepsilon^2)^{q/2}} \sin\left( \frac{1}{2} t \varepsilon^{-1} |k-k_0|^3 \tilde{\gamma}_{k_0} (k) \right) \chi_\mathfrak{K}(k-k_0) \right] \right\|_{L_2(\mathbb{R}) \to L_2(\mathbb{R})}  
			\\
			\asymp \begin{dcases}
				\frac{\varepsilon^{2q/3} |t|^{q/3}}{(|\tilde{\gamma}_{k_0} (k_0)|^{-2/3} + \varepsilon^{4/3} |t|^{2/3})^{q/2}}, & \text{for} \quad 0 \le q \le 3 \quad (\text{and} \quad 0 < \varepsilon |t|^{-1} \le \tilde{\mathfrak{e}}),
				\\
				\vphantom{\frac{0}{0}} |\tilde{\gamma}_{k_0} (k_0)| \, \varepsilon^2 |t|, & \text{for} \quad q > 3,
			\end{dcases}
		\end{multlined}
		\\
		\label{lemma2_est_3}
		&\begin{multlined}[c][0.8\textwidth]
		\left\| \left[ \frac{\varepsilon^{r}}{((k-k_0)^2 + \varepsilon^2)^{r/2}} |k-k_0|^{-1}  \sin\left( \frac{1}{2}  t \varepsilon^{-1} |k-k_0|^3 \tilde{\gamma}_{k_0} (k) \right) \chi_\mathfrak{K}(k-k_0) \right] \right\|_{L_2(\mathbb{R}) \to L_2(\mathbb{R})}  
		\\
		\asymp \begin{dcases}
			\cfrac{\varepsilon^{(2r-1)/3}|t|^{(r+1)/3}}{|\tilde{\gamma}_{k_0} (k_0)|^{-1/3} (|\tilde{\gamma}_{k_0} (k_0)|^{-2/3} + \varepsilon^{4/3} |t|^{2/3})^{r/2}}, & \text{for} \quad -1 \le r \le 2 \; (\text{and} \;\; 0 < \varepsilon |t|^{-1} \le \tilde{\mathfrak{e}}),
			\\
			\vphantom{\frac{0}{0}}  |\tilde{\gamma}_{k_0} (k_0)| \, \varepsilon |t|, & \text{for} \quad r > 2.
		\end{dcases}
		\end{multlined}
	\end{align}
	Here the notation $X \asymp Y$ means that $c_1 Y \le X \le c_2 Y$ with some constants $c_1$ and $c_2$, which may depend only on $q$ and $r$\textup{;} $\mathfrak{e} \coloneqq (2\pi)^{-1/2}  |\gamma_{k_0} (k_0)|^{1/2} \kappa^2$, $\tilde{\mathfrak{e}} \coloneqq (2\pi)^{-1}  |\tilde{\gamma}_{k_0} (k_0)| \kappa^{3}$.
\end{lemma}
\begin{proof}
	Let us prove estimate~\eqref{lemma2_est_1}. Estimates~\eqref{lemma2_est_2}, \eqref{lemma2_est_3} can be proved in a similar way. We have
	\begin{equation}
		\label{lemma2_norm}
		\begin{multlined}[c][0.8\textwidth]
			\left\| \left[ \varepsilon^{q}((k-k_0)^2 + \varepsilon^2)^{-q/2} \sin\left( \frac{1}{2} t \varepsilon^{-2} (k-k_0)^4 \gamma_{k_0} (k) \right) \chi_\mathfrak{K}(k-k_0) \right] \right\|_{L_2(\mathbb{R}) \to L_2(\mathbb{R})} \\ = \left\| \varepsilon^{q}((k-k_0)^2 + \varepsilon^2)^{-q/2} \sin\left( \frac{1}{2} t \varepsilon^{-2} (k-k_0)^4 \gamma_{k_0} (k) \right) \chi_\mathfrak{K}(k-k_0) \right\|_{L_\infty}.
		\end{multlined}
	\end{equation}
	Let us estimate the norm in the right-hand side of~\eqref{lemma2_norm}. Introduce the functions
	\begin{equation*}
		h_1(y) = \left\{
		\begin{aligned}
			&\tfrac{2}{\pi}y, & &\text{for} \, y \in [0, \pi/2],\\
			&\tfrac{2}{\pi}(\pi - y), & &\text{for} \, y \in (\pi/2, \pi],
		\end{aligned}
		\right. 
		\qquad
		h_2(y) = \left\{
		\begin{aligned}
			&y, & &\text{for} \, y \in [0, \pi/2],\\
			&\pi - y, & &\text{for} \, y \in (\pi/2, \pi],
		\end{aligned}
		\right.
	\end{equation*}
	and $1$-periodically extend them to $\mathbb{R}$.
	We have 
	\begin{equation*}
		h_1 (y) \le |\sin y| \le h_2 (y).
	\end{equation*}
	We shall estimate maxima of the functions
	\begin{equation}
		\label{lemma2_f1}
		\varepsilon^{q}((k-k_0)^2 + \varepsilon^2)^{-q/2} h_j \left( \frac{1}{2} t \varepsilon^{-2} (k-k_0)^4 \gamma_{k_0} (k) \right) \chi_\mathfrak{K}(k-k_0), \qquad j=1,2,
	\end{equation}
	from below (for $j=1$) and from above (for $j=2$). Since $\varepsilon^{q}(k^2 + \varepsilon^2)^{-q/2}$ is monotone decreasing in $k > 0$, and $h_1$, $h_2$ are periodic, maxima of functions~\eqref{lemma2_f1} are reached on the set
	\begin{equation*}
		K = \left\{k \colon \frac{1}{2} |t| \varepsilon^{-2} (k-k_0)^4 |\gamma_{k_0} (k)|  \le \frac{\pi}{2} \right\}.
	\end{equation*} 
	
	Consider the sets
	\begin{align}
		\label{K1'}
		K'_1 &= \left\{k \colon \frac{3}{4} |t| \varepsilon^{-2} \cdot |\gamma_{k_0} (k_0)| (k-k_0)^4  \le \frac{\pi}{2} \right\},\\
		\label{K2'}
		K'_2 &= \left\{k \colon \frac{1}{4} |t| \varepsilon^{-2} \cdot |\gamma_{k_0} (k_0)| (k-k_0)^4  \le \frac{\pi}{2} \right\}.
	\end{align} 
	From~\eqref{delta_fix} it follows that $K'_1 \subset K \subset K'_2$. Let $\check{k} \coloneqq \max K$ and $\check{k}'_{1,2} \coloneqq \max K'_{1,2}$. We have
	\begin{equation*}
		\check{k}'_1 = k_0 + (2\pi/3)^{1/4} |\gamma_{k_0} (k_0)|^{-1/4}  \varepsilon^{1/2} |t|^{-1/4}, \quad \check{k}'_2 = k_0 + (2\pi)^{1/4} |\gamma_{k_0} (k_0)|^{-1/4}  \varepsilon^{1/2} |t|^{-1/4}
	\end{equation*}
	(maxima are reached when we have equalities in~\eqref{K1'}, \eqref{K2'}), and
	\begin{equation}
		\label{k_check_est}
		|\check{k}'_1 - k_0| \le |\check{k} - k_0| \le |\check{k}'_2 - k_0|.
	\end{equation}
	Consider function~\eqref{lemma2_f1} with $j=2$. The right inequality in~\eqref{delta_fix} implies
	\begin{multline*}
		\varepsilon^{q}((k-k_0)^2 + \varepsilon^2)^{-q/2} h_2 \left( \frac{1}{2} t \varepsilon^{-2} (k-k_0)^4 \gamma_{k_0} (k) \right) \chi_\mathfrak{K}(k-k_0) \\ = \varepsilon^{q}((k-k_0)^2 + \varepsilon^2)^{-q/2} \cdot \frac{1}{2} |t| \varepsilon^{-2} (k-k_0)^4 |\gamma_{k_0} (k)| \chi_\mathfrak{K}(k-k_0) \\ \le \varepsilon^{q}((k-k_0)^2 + \varepsilon^2)^{-q/2} \cdot \frac{3}{4} |t| \varepsilon^{-2} |\gamma_{k_0} (k_0)| (k-k_0)^4 \chi_\mathfrak{K}(k-k_0).
	\end{multline*}
	The function $k^4 (k^2 + \varepsilon^2)^{-q/2}$ is even, and monotonically increases in $k \ge 0$ for $0 \le q \le 4$. We substitute $k = \check{k}$ in the right-hand side of the inequality obtained above, and estimate it using the right inequality in~\eqref{k_check_est}:
	\begin{equation*}
		\begin{multlined}[c][0.9\textwidth]
			\varepsilon^{q}((\check{k}-k_0)^2 + \varepsilon^2)^{-q/2} \cdot \frac{3}{4} |t| \varepsilon^{-2} |\gamma_{k_0} (k_0)| (\check{k}-k_0)^4 
			\\
			\le 
			\varepsilon^{q}\left( (2\pi)^{1/2} |\gamma_{k_0} (k_0)|^{-1/2}  \varepsilon |t|^{-1/2} + \varepsilon^2\right) ^{-q/2} \cdot \frac{3}{4} |t| \varepsilon^{-2} |\gamma_{k_0} (k_0)| \cdot (2\pi) |\gamma_{k_0} (k_0)|^{-1}  \varepsilon^2 |t|^{-1}
			\\
			= \frac{3 \pi}{2}  \cdot \varepsilon^{q}\left( (2\pi)^{1/2} |\gamma_{k_0} (k_0)|^{-1/2}  \varepsilon |t|^{-1/2} + \varepsilon^2\right)^{-q/2} = \frac{3 \pi}{2}  \cdot \frac{\varepsilon^{q/2} |t|^{q/4}}{\bigl((2\pi)^{1/2} |\gamma_{k_0} (k_0)|^{-1/2}  + \varepsilon |t|^{1/2}\bigr)^{q/2}}, 
			\\
			0 < \varepsilon |t|^{-1/2} \le \mathfrak{e}, \quad 0 \le q \le 4.
		\end{multlined}
	\end{equation*}
	Here the inclusion $K'_2 \subset (k_0 - \kappa, k_0 + \kappa)$ (which follows from $0 < \varepsilon |t|^{-1/2} \le \mathfrak{e}$) is taken into account. For $q > 4$ the function $k^4 (k^2 + \varepsilon^2)^{-q/2}$, $k \ge 0$, has the maximum $16 q^{-q/2} (q-4)^{q/2-2} \varepsilon^{4-q}$. Then
	\begin{multline*}
		\varepsilon^{q}((k-k_0)^2 + \varepsilon^2)^{-q/2} \cdot \frac{3}{4} |t| \varepsilon^{-2} |\gamma_{k_0} (k_0)| (k-k_0)^4 \\ \le 12 q^{-q/2} (q-4)^{q/2-2} |\gamma_{k_0} (k_0)| \varepsilon^{2} |t|, 
		\qquad	q > 4.
	\end{multline*}
	As a result, we have obtained the upper estimate in~\eqref{lemma2_est_1}.
	
	Now we consider function~\eqref{lemma2_f1} with $j=1$. From the left inequality in~\eqref{delta_fix} it follows that
	\begin{multline*}
		\varepsilon^{q}((k-k_0)^2 + \varepsilon^2)^{-q/2} h_1 \left( \frac{1}{2} t \varepsilon^{-2} (k-k_0)^4 \gamma_{k_0} (k) \right) \chi_\mathfrak{K}(k-k_0) \\ \ge \varepsilon^{q}((k-k_0)^2 + \varepsilon^2)^{-q/2} \cdot \frac{1}{2 \pi} |t| \varepsilon^{-2} |\gamma_{k_0} (k_0)| (k-k_0)^4 \chi_\mathfrak{K}(k-k_0).
	\end{multline*}
	For $0 \le q \le 4$ we substitute $k = \check{k}$ in the right-hand side of this inequality, and estimate it using the left inequality in~\eqref{k_check_est}: 
	\begin{multline*}
			\varepsilon^{q}((\check{k}-k_0)^2 + \varepsilon^2)^{-q/2} \cdot \frac{1}{2 \pi} |t| \varepsilon^{-2} |\gamma_{k_0} (k_0)| (\check{k}-k_0)^4 \\ 
			\ge \frac{1}{3}  \cdot \frac{\varepsilon^{q/2} |t|^{q/4}}{\bigl((2\pi/3)^{1/2} |\gamma_{k_0} (k_0)|^{-1/2}  + \varepsilon |t|^{1/2}\bigr)^{q/2}},
			\qquad
			0 < \varepsilon |t|^{-1/2} \le \mathfrak{e}, \quad 0 \le q \le 4,
	\end{multline*}	
	and for $q > 4$ we have
	\begin{multline*}
		\max_{k} \varepsilon^{q}((k-k_0)^2 + \varepsilon^2)^{-q/2} h_1 \left( \frac{1}{2} t \varepsilon^{-2} (k-k_0)^4 \gamma_{k_0} (k) \right) \chi_\mathfrak{K}(k-k_0) \\ \ge \frac{8}{\pi} q^{-q/2} (q-4)^{q/2-2} |\gamma_{k_0} (k_0)| \varepsilon^{2} |t|, \qquad q > 4.
	\end{multline*}
	Thereby, the lower estimate in~\eqref{lemma2_est_1} is proved.
\end{proof}

\section{Main results of the paper}
\label{main_results_section}
In this section, we formulate the main results of the paper. We give detailed proofs in the cases of Conditions~\ref{cond1} and~\ref{cond3}, and we only formulate the results in the cases of Conditions~\ref{cond2} and~\ref{cond4}.

\subsection{The case where Condition~\ref{cond1} is fulfilled.}
Let $\varepsilon > 0$ be a small parameter. In $L_2(\mathbb{R})$, we consider the operator formally defined by the differential expression
\begin{equation}
	\label{A_eps}
	A_\varepsilon = - (\omega^\varepsilon)^{-1} \frac{d}{d x} g^\varepsilon \frac{d}{d x}(\omega^\varepsilon)^{-1}, \quad g = \check{g} \omega^2.
\end{equation}
Here $\check{g}$ and $\omega$ are $1$-periodic functions satisfying conditions~\eqref{g_check_cond}, \eqref{omega_cond1}, and \eqref{omega_cond2}. The precise definition of the operator $A_\varepsilon$ is given in terms of the corresponding quadratic form (cf.~\eqref{a_form_2}). Operators~\eqref{A} and~\eqref{A_eps} satisfy the following relation:
\begin{equation*}
	A_\varepsilon = \varepsilon^{-2} T_\varepsilon^* A T_\varepsilon,
\end{equation*}
where $T_\varepsilon$ is the operator of scaling transformation: $(T_\varepsilon u)(x) = \varepsilon^{1/2} u(\varepsilon x)$.

Now it is convenient to realize points of the circle in~\eqref{Psi_l} in the following way:
\begin{equation*}
	\Psi_l \colon L_2(\mathbb{R}) \to L_2(\widetilde{\Omega}_{l-s+1}),  \qquad l \ge s.
\end{equation*}
Recall that $\widetilde{\Omega}_j$, $j \in \mathbb{N}$, were defined in~\eqref{Brillouin_zones}. Let $f \in L_2(\mathbb{R})$. We study the behavior of the solution $u_\varepsilon (x,t)$, $\varepsilon \to 0$, of the following Cauchy problem for the nonstationary Schr\"{o}dinger equation
\begin{equation}
	\label{cond1_Cauchy_problem_Schrod}
	\left\{
		\begin{aligned}
			&i \frac{\partial}{\partial t} u_\varepsilon (x,t) = (A_\varepsilon u_\varepsilon)(x,t),\\
			&u_\varepsilon (x,0) = (\Upsilon^{(+)}_\varepsilon f)(x),
		\end{aligned}
	 \right.
\end{equation} 
where
\begin{align*}
	&(\Upsilon^{(+)}_\varepsilon f)(x) \coloneqq (2 \pi)^{-1/2} \int_\mathbb{R} (\Phi f) (k) \sum_{j=s}^{\infty} e^{ikx} \varphi_j(x/\varepsilon, \varepsilon k) \chi_{\widetilde{\Omega}_{j-s+1}} (\varepsilon k) \, dk,
	\\
	&\Upsilon^{(+)}_\varepsilon f \in \mathcal{E}_{A_\varepsilon}[\varepsilon^{-2} \sigma_0, \infty) L_2(\mathbb{R}).
\end{align*}

In $L_2(\mathbb{R})$, we consider the operator $A^{\mathrm{hom}}_0 =  - b_0 \frac{d^2}{d x^2}$, $\Dom A^{\mathrm{hom}}_0 = H^2(\mathbb{R})$, which is called \emph{the effective operator at the left edge $\sigma_{0} = E(0)$ of the band $R(E)$}. Let $u_0 (x,t)$ be the solution of the corresponding "homogenized" problem
\begin{equation}
	\label{cond1_Cauchy_problem_Schrod_homog}
	\left\{
	\begin{aligned}
		&i \frac{\partial}{\partial t} u_0 (x,t) = (A^{\mathrm{hom}}_0 u_0)(x,t) ,\\
		&u_0 (x,0) = f (x).
	\end{aligned}
	\right.
\end{equation}
The solutions of problems~\eqref{cond1_Cauchy_problem_Schrod} and~\eqref{cond1_Cauchy_problem_Schrod_homog} can be represented as follows:
\begin{equation}
	\label{u_represent}
	u_\varepsilon = e^{-i t A_\varepsilon}  \Upsilon^{(+)}_\varepsilon f, \qquad u_0 = e^{-i t A^{\mathrm{hom}}_0} f,
\end{equation}
where $\Upsilon^{(+)}_\varepsilon \coloneqq \sum_{j=s}^\infty T_\varepsilon^* \Psi_j^* R_{j-s+1}  \Phi T_\varepsilon$, and  $R_l$ is the operator of restriction to $\widetilde{\Omega}_l$.
\begin{thrm}
	\label{cond1_Schrod_thrm}
	Suppose that Condition~\textup{\ref{cond1}} and~\eqref{gamma_ne_0} are fulfilled. Let $u_\varepsilon$ be the solution of problem~\eqref{cond1_Cauchy_problem_Schrod}, and let $u_0$ be the solution of problem~\eqref{cond1_Cauchy_problem_Schrod_homog}. Let $t \ne 0$, $f \in H^q(\mathbb{R})$, $0 \le q \le 2$. We have
	\begin{equation}
		\label{cond1_Schrod_thrm_est}
		\| u_\varepsilon(\cdot, t) - e^{-i t \varepsilon^{-2} \sigma_0} \varphi_0^{\varepsilon}  u_0(\cdot, t)\|_{L_2(\mathbb{R})} \le C (1+|t|^{q/4}) \varepsilon^{q/2} \|f\|_{H^q(\mathbb{R})}, \qquad 0 < \varepsilon |t|^{-1/2} \le \mathfrak{e},
	\end{equation}
	with the constant $C = C(q, \kappa, \| \varphi_{0} \|_{L_\infty}, |\gamma_{0} (0)|, \|\theta_0\|_{M_\kappa(0)})$.
\end{thrm} 
\begin{proof}
	By~\eqref{u_represent}, estimate~\eqref{cond1_Schrod_thrm_est} can be reformulated in the operator terms:
	\begin{equation}
		\label{case1_Schr_thrm_est_operat}
		\Bigl\| e^{-i t A_\varepsilon} \Upsilon^{(+)}_\varepsilon - e^{-i t \varepsilon^{-2} \sigma_0} [\varphi_0^{\varepsilon}] e^{-i t A^{\mathrm{hom}}_0} \Bigr\|_{H^q (\mathbb{R}) \to L_2(\mathbb{R})} \le C (1+|t|^{q/4}) \varepsilon^{q/2},
	\end{equation}
	where $t \ne 0$, $0 < \varepsilon |t|^{-1/2} \le \mathfrak{e}$, $0 \le q \le 2$. Thus, our aim is to prove~\eqref{case1_Schr_thrm_est_operat}.  Since
	\begin{equation}
		\label{Sobolev_isomorhism}
		\left.
		\begin{aligned}
			&\text{the operator } (-d_x^2 + I)^{q/2} \text{ is an isometric isomorphism} \\
			&\text{of the Sobolev space } H^q(\mathbb{R}) \text{ onto } L_2(\mathbb{R}),
		\end{aligned} 
		\quad \right\rbrace		
	\end{equation}
	we have
	\begin{multline*}
		\Bigl\| e^{-i t A_\varepsilon} \Upsilon^{(+)}_\varepsilon - e^{-i t \varepsilon^{-2} \sigma_0} [\varphi_0^{\varepsilon}] e^{-i t A^{\mathrm{hom}}_0} \Bigr\|_{H^q (\mathbb{R}) \to L_2(\mathbb{R})} \\=  \Bigl\| \Bigl(e^{-i t A_\varepsilon} \Upsilon^{(+)}_\varepsilon - e^{-i t \varepsilon^{-2} \sigma_0} [\varphi_0^{\varepsilon}] e^{-i t A^{\mathrm{hom}}_0}\Bigr) (- d_x^2 + I)^{-q/2} \Bigr\|_{L_2 (\mathbb{R}) \to L_2(\mathbb{R})}.
	\end{multline*}
	From the unitarity of the scaling transformation it directly follows that 
	\begin{equation}
		\label{thrm1_proof_scale_transf}
		\begin{multlined}[c][0.9\textwidth]
			\Bigl\| \Bigl(e^{-i t A_\varepsilon} \Upsilon^{(+)}_\varepsilon -  e^{-i t \varepsilon^{-2} \sigma_0}[\varphi_0^{\varepsilon}] e^{-i t A^{\mathrm{hom}}_0}\Bigr) (- d_x^2 + I)^{-q/2} \Bigr\|_{L_2 (\mathbb{R}) \to L_2(\mathbb{R})} 
			\\ 
			= \Bigl\| \Bigl( e^{-i t \varepsilon^{-2} A} \Upsilon^{(+)} - e^{-i t \varepsilon^{-2} \sigma_0} [\varphi_0] e^{-i t \varepsilon^{-2} A^{\mathrm{hom}}_0} \Bigr) \varepsilon^{q} (- d_x^2 + \varepsilon^{2}I)^{-q/2} \Bigr\|_{L_2 (\mathbb{R}) \to L_2(\mathbb{R})},
		\end{multlined}
	\end{equation}
	where $\Upsilon^{(+)} \coloneqq \sum_{j=s}^\infty \Psi_j^* R_{j-s+1} \Phi$. Introduce the projection $F_\mathfrak{K} = \Phi^* [\chi_\mathfrak{K}] \Phi$. Obviously,
	\begin{equation}
		\| \varepsilon^q (- d_x^2 + \varepsilon^{2}I)^{-q/2} (I - F_\mathfrak{K}) \|_{L_2 (\mathbb{R}) \to L_2(\mathbb{R})} \le \kappa^{-q} \varepsilon^q,
	\end{equation}
	whence
	\begin{equation}
		\label{thrm_(I-F)_est}
		\begin{multlined}[c][0.9\textwidth]
			\Bigl\| \Bigl( e^{-i t \varepsilon^{-2} A} \Upsilon^{(+)} - e^{-i t \varepsilon^{-2} \sigma_0} [\varphi_0] e^{-i t \varepsilon^{-2} A^{\mathrm{hom}}_0}\Bigr) \varepsilon^{q} (- d_x^2 + \varepsilon^{2}I)^{-q/2} (I - F_\mathfrak{K}) \Bigr\|_{L_2 (\mathbb{R}) \to L_2(\mathbb{R})} \\ \le (1 + \|\varphi_0\|_{L_\infty} )  \kappa^{-q} \varepsilon^q.
		\end{multlined}
	\end{equation}
	Consider the operator $\bigl( e^{-i t \varepsilon^{-2} A} \Upsilon^{(+)} - e^{-i t \varepsilon^{-2} \sigma_0} [\varphi_0] e^{-i t \varepsilon^{-2} A^{\mathrm{hom}}_0}\bigr) \varepsilon^{q} (- d_x^2 + \varepsilon^{2}I)^{-q/2} F_\mathfrak{K}$. The following identities are true: 
	\begin{align}
		\label{thrm_f1}
		e^{-i t \varepsilon^{-2} A} \Upsilon^{(+)} \ \varepsilon^{q} (- d_x^2 + \varepsilon^{2}I)^{-q/2} F_\mathfrak{K} &= \Psi^* [e^{-i t \varepsilon^{-2} E(k)} \varepsilon^{q} (k^2 + \varepsilon^{2})^{-q/2} \chi_\mathfrak{K}(k)] \Phi,\\
		\label{thrm_f2}
		[\varphi_0] e^{-i t \varepsilon^{-2} A^{\mathrm{hom}}_0} \varepsilon^{q} (- d_x^2 + \varepsilon^{2}I)^{-q/2} F_\mathfrak{K} &= [\varphi_0] \Phi^* [e^{-i t \varepsilon^{-2} b_0 k^2} \varepsilon^{q} (k^2 + \varepsilon^{2})^{-q/2} \chi_\mathfrak{K}(k)] \Phi.
	\end{align}
	By~\eqref{lemma1_est_1}, we have
	\begin{equation}
		\label{thrm1_proof_lemma1_appl}
		\| (\Psi^* - [\varphi_0] \Phi^*)  [e^{-i t \varepsilon^{-2} b_0 k^2} \varepsilon^{q} (k^2 + \varepsilon^{2})^{-q/2} \chi_\mathfrak{K}(k)] \Phi  \| \le C \varepsilon^{\min(1,q)}.
	\end{equation}
	It remains to estimate $\| \Psi^* [(e^{-i t \varepsilon^{-2} E(k)} - e^{-i t \varepsilon^{-2} \sigma_0} e^{-i t \varepsilon^{-2} b_0 k^2})  \varepsilon^{q} (k^2 + \varepsilon^{2})^{-q/2} \chi_\mathfrak{K}(k)] \Phi\|_{L_2 (\mathbb{R}) \to L_2(\mathbb{R})}$. 
	From~\eqref{E_powerseries_1} it is seen that
	\begin{equation}
		\label{exp-exp_sin}
		\begin{multlined}[c][0.8\textwidth]
			e^{-i t \varepsilon^{-2} E(k)} - e^{-i t \varepsilon^{-2} \sigma_0} e^{-i t \varepsilon^{-2} b_0 k^2} = e^{-i t \varepsilon^{-2} (\sigma_0 + b_0 k^2)} (e^{-i t \varepsilon^{-2} k^4 \gamma_0 (k)} - 1)
			\\ 
			= -e^{-i t \varepsilon^{-2} (\sigma_0 + b_0 k^2)} e^{-\frac{1}{2} i t \varepsilon^{-2} k^4 \gamma_0 (k)} \cdot 2i \sin\left( \frac{1}{2} t \varepsilon^{-2} k^4 \gamma_0 (k) \right).
		\end{multlined}
	\end{equation}
	Applying upper estimate~\eqref{lemma2_est_1} and taking into account that $| e^{-i t \varepsilon^{-2} (\sigma_0 + b_0 k^2)} | = 1$, $\bigl| e^{-\frac{1}{2} i t \varepsilon^{-2} k^4 \gamma_0 (k)}\bigr| = 1$, we get
	\begin{equation}
		\label{thrm_f3}
		\begin{multlined}[c][0.8\textwidth]
			\left\| \Psi^* [(e^{-i t \varepsilon^{-2} E(k)} - e^{-i t \varepsilon^{-2} \sigma_0} e^{-i t \varepsilon^{-2} b_0 k^2})  \varepsilon^{q} (k^2 + \varepsilon^{2})^{-q/2} \chi_\mathfrak{K}(k)] \Phi \right\|_{L_2 (\mathbb{R}) \to L_2(\mathbb{R})} \\ \le  \frac{C \varepsilon^{q/2} |t|^{q/4}}{(|\gamma_{k_0} (k_0)|^{-1/2} + \varepsilon |t|^{1/2})^{q/2}}, \qquad 0 < \varepsilon |t|^{-1/2} \le \mathfrak{e}.
		\end{multlined}
	\end{equation}
	Combining~\eqref{thrm1_proof_scale_transf}, \eqref{thrm_(I-F)_est}--\eqref{thrm1_proof_lemma1_appl}, and \eqref{thrm_f3}, we arrive at the required estimate~\eqref{case1_Schr_thrm_est_operat}.
\end{proof}

Let $f, g \in L_2(\mathbb{R})$. Now we consider the behavior of the solution $v_\varepsilon (x,t)$, $\varepsilon \to 0$, of the Cauchy problem for the hyperbolic equation
\begin{equation}
	\label{cond1_Cauchy_problem_hyperb}
	\left\{
	\begin{aligned}
		&\frac{\partial^2}{\partial t^2}  v_\varepsilon (x,t) = - (A_\varepsilon v_\varepsilon)(x,t) + \varepsilon^{-2} \sigma_0 v_\varepsilon(x,t),\\
		&v_\varepsilon (x,0) = (\Upsilon^{(+)}_\varepsilon f)(x), \; (\partial_t v_\varepsilon) (x,0) = (\Upsilon^{(+)}_\varepsilon g)(x).
	\end{aligned}
	\right.
\end{equation}
Let $v_0 (x,t)$ be the solution of the corresponding "homogenized" problem
\begin{equation}
	\label{cond1_Cauchy_problem_hyperb_homog}
	\left\{
	\begin{aligned}
		&\frac{\partial^2}{\partial t^2} v_0 (x,t) = - (A^{\textrm{hom}}_0 v_0)(x,t) ,\\
		&v_0 (x,0) = f (x), \; (\partial_t v_0) (x,0) = g (x).
	\end{aligned}
	\right.
\end{equation}
The solutions of problems~\eqref{cond1_Cauchy_problem_hyperb} and~\eqref{cond1_Cauchy_problem_hyperb_homog} satisfy the following representations:
\begin{equation}
	\label{v_represent}
	\left.
	\begin{aligned}
		&v_\varepsilon = \cos \bigl(t A_{\sigma_{0}, \varepsilon}^{1/2}\bigr) \Upsilon^{(+)}_\varepsilon f + A_{\sigma_{0}, \varepsilon}^{-1/2} \sin \bigl(t A_{\sigma_{0}, \varepsilon}^{1/2}\bigr) \Upsilon^{(+)}_\varepsilon g,
		\\
		&v_0 = \cos (t (A^{\mathrm{hom}}_0)^{1/2}) f + (A^{\mathrm{hom}}_0)^{-1/2} \sin (t (A^{\mathrm{hom}}_0)^{1/2}) g,
	\end{aligned}
	\quad\right\rbrace 
\end{equation}
where $A_{\sigma_{0}, \varepsilon} \coloneqq (A_\varepsilon  - \varepsilon^{-2}\sigma_0) \mathcal{E}_{A_\varepsilon}[\varepsilon^{-2} \sigma_0, \infty)$.

\begin{thrm}
	\label{cond1_hyperb_thrm}
	Suppose that Condition~\textup{\ref{cond1}} and~\eqref{gamma_ne_0} are fulfilled. Let $v_\varepsilon$ be the solution of problem~\eqref{cond1_Cauchy_problem_hyperb}, and let $v_0$ be the solution of problem~\eqref{cond1_Cauchy_problem_hyperb_homog}. Let $t \ne 0$, $f \in H^{q}(\mathbb{R})$, $g \in H^{r}(\mathbb{R})$. We have
	\begin{align}
		\label{cond1_hyperb_thrm_est}
		&\begin{multlined}[c][0.8\textwidth]
			\| v_\varepsilon(\cdot, t) - \varphi_0^{\varepsilon}  v_0(\cdot, t)\|_{L_2(\mathbb{R})} \\ \le C \bigl( (1+|t|^{q/3}) \varepsilon^{2q/3} \|f\|_{H^{q}(\mathbb{R})} + (1+|t|^{(r+1)/3}) \varepsilon^{(2r+2)/3}  \|g\|_{H^{r}(\mathbb{R})} \bigr), \\ 0 < \varepsilon |t|^{-1} \le \tilde{\mathfrak{e}}, \quad 0 < \varepsilon \le 1, \quad 0 \le q \le 3/2, \quad 0 \le r \le 1/2,
		\end{multlined}
		\\
		\label{cond1_hyperb_thrm_est_r<0}
		&\begin{multlined}[c][0.8\textwidth]
			\| v_\varepsilon(\cdot, t) - \varphi_0^{\varepsilon}  v_0(\cdot, t)\|_{L_2(\mathbb{R})} \\ \le C \bigl( (1+|t|^{q/3}) \varepsilon^{2q/3} \|f\|_{H^{q}(\mathbb{R})} + ((1+|t|^{(r+1)/3}) \varepsilon^{(2r+2)/3} + |t|^{1/3} \varepsilon^{2/3}) \|g\|_{H^{r}(\mathbb{R})} \bigr), \\ 0 < \varepsilon |t|^{-1} \le \tilde{\mathfrak{e}}, \quad 0 < \varepsilon \le 1, \quad 0 \le q \le 3/2, \quad -1 \le r < 0,
		\end{multlined}
	\end{align}
	with the constants $C = C(q, r, \kappa, \alpha_0, \beta_0, \beta_1, \sigma_{0}, b_0, \| \varphi_{0} \|_{L_\infty}, |\tilde{\gamma}_{0} (0)|, \|\theta_0\|_{M_\kappa(0)})$.
\end{thrm}
\begin{proof}
	By~\eqref{v_represent}, we have to prove the estimates
	\begin{align}
		\label{cos_norm}
		&\begin{multlined}[c][0.8\textwidth]
		\Bigl\| \cos \bigl(t A_{\sigma_{0}, \varepsilon}^{1/2}\bigr) \Upsilon^{(+)}_\varepsilon-  [\varphi_0^{\varepsilon}] \cos (t (A^{\mathrm{hom}}_0)^{1/2}) \Bigr\|_{H^{q} (\mathbb{R}) \to L_2(\mathbb{R})}  \le C (1+|t|^{q/3}) \varepsilon^{2q/3}, 
		\\ 0 \le q \le 3/2,		
		\end{multlined}
		\\
		\label{sin_norm_r>0}
		&\begin{multlined}[c][0.8\textwidth]
			\Bigl\| A_{\sigma_{0}, \varepsilon}^{-1/2} \sin \bigl( t A_{\sigma_{0}, \varepsilon}^{1/2}\bigr) \Upsilon^{(+)}_\varepsilon  -  [\varphi_0^{\varepsilon}] (A^{\mathrm{hom}}_0)^{-1/2} \sin (t (A^{\mathrm{hom}}_0)^{1/2}) \Bigr\|_{H^{r} (\mathbb{R}) \to L_2(\mathbb{R})} \\
			\le C (1+|t|^{(r+1)/3}) \varepsilon^{(2r+2)/3}, \qquad 0 \le r \le 1/2,
		\end{multlined}
		\\
		\label{sin_norm_r<0}
		&\begin{multlined}[c][0.8\textwidth]
			\Bigl\| A_{\sigma_{0}, \varepsilon}^{-1/2} \sin \bigl( t A_{\sigma_{0}, \varepsilon}^{1/2}\bigr) \Upsilon^{(+)}_\varepsilon  -  [\varphi_0^{\varepsilon}] (A^{\mathrm{hom}}_0)^{-1/2} \sin (t (A^{\mathrm{hom}}_0)^{1/2}) \Bigr\|_{H^{r} (\mathbb{R}) \to L_2(\mathbb{R})} \\
			\le C ((1+|t|^{(r+1)/3}) \varepsilon^{(2r+2)/3} + |t|^{1/3} \varepsilon^{2/3}), \qquad -1 \le r < 0,
		\end{multlined}
	\end{align}
	where $t \ne 0$, $0 < \varepsilon |t|^{-1} \le \tilde{\mathfrak{e}}$. Using~\eqref{Sobolev_isomorhism} and the unitarity of the scaling transformation, we have
	\begin{equation}
		\label{cos_norm_reduction}
		\begin{multlined}[c][0.8\textwidth]
			\Bigl\| \cos \bigl(t A_{\sigma_{0}, \varepsilon}^{1/2} \bigr) \Upsilon^{(+)}_\varepsilon -  [\varphi_0^{\varepsilon}] \cos (t (A^{\mathrm{hom}}_0)^{1/2}) \Bigr\|_{H^{q} (\mathbb{R}) \to L_2(\mathbb{R})}
			\\
			= \Bigl\| \Bigl( \cos\bigl(t \varepsilon^{-1} A_{\sigma_{0}}^{1/2}\bigr) \Upsilon^{(+)} -  [\varphi_0] \cos(t \varepsilon^{-1} (A^{\mathrm{hom}}_0)^{1/2}) \Bigr) \varepsilon^{q} (- d_x^2 + \varepsilon^{2}I)^{-q/2} \Bigr\|_{L_2 (\mathbb{R}) \to L_2(\mathbb{R})}
		\end{multlined}
	\end{equation}
	and
	\begin{equation}
		\label{sin_norm_reduction}
		\begin{aligned}
				&\Bigl\| A_{\sigma_{0}, \varepsilon}^{-1/2} \sin \bigl( t A_{\sigma_{0}, \varepsilon}^{1/2}\bigr) \Upsilon^{(+)}_\varepsilon -  [\varphi_0^{\varepsilon}] (A^{\mathrm{hom}}_0)^{-1/2} \sin (t (A^{\mathrm{hom}}_0)^{1/2}) \Bigr\|_{H^{r}(\mathbb{R}) \to L_2(\mathbb{R})}
			\\
			&\begin{multlined}[c][0.9\textwidth]
				= \Bigl\| \varepsilon \Bigl( A_{\sigma_{0}}^{-1/2} \sin \bigl(t \varepsilon^{-1} A_{\sigma_{0}}^{1/2}\bigr) \Upsilon^{(+)}  \\ -  [\varphi_0] (A^{\mathrm{hom}}_0)^{-1/2} \sin (t \varepsilon^{-1} (A^{\mathrm{hom}}_0)^{1/2}) \Bigr) \varepsilon^{r} (- d_x^2 + \varepsilon^{2}I)^{-r/2} \Bigr\|_{L_2(\mathbb{R}) \to L_2(\mathbb{R})},
			\end{multlined}
		\end{aligned}
	\end{equation}
	where $A_{\sigma_0} \coloneqq (A  - \sigma_0) \mathcal{E}_{A}[\sigma_0, \infty)$.
	
	Let us estimate~\eqref{cos_norm_reduction}. Similarly to the proof of Theorem~\ref{cond1_Schrod_thrm}, one obtains
	\begin{gather*}
	\begin{multlined}[c][0.8\textwidth]
		\Bigl\| \Bigl( \cos\bigl(t \varepsilon^{-1} A_{\sigma_{0}}^{1/2}\bigr) \Upsilon^{(+)} - [\varphi_0] \cos(t \varepsilon^{-1} (A^{\mathrm{hom}}_0)^{1/2}) \Bigr) \varepsilon^{q} (- d_x^2 + \varepsilon^{2}I)^{-q/2} (I - F_\mathfrak{K}) \Bigr\|_{L_2 (\mathbb{R}) \to L_2(\mathbb{R})} \\ \le (1 + \|\varphi_0\|_{L_\infty} ) \kappa^{-q} \varepsilon^q,
	\end{multlined}
	\\
		\cos\bigl(t \varepsilon^{-1} A_{\sigma_{0}}^{1/2}\bigr) \Upsilon^{(+)} \ \varepsilon^{q} (- d_x^2 + \varepsilon^{2}I)^{-q/2} F_\mathfrak{K} = \Psi^* [\cos\bigl(t \varepsilon^{-1} (E(k) - \sigma_0)^{1/2}\bigr) \varepsilon^{q} (k^2 + \varepsilon^{2})^{-q/2} \chi_\mathfrak{K}(k)] \Phi,
		\\
		[\varphi_0] \cos\bigl(t \varepsilon^{-1} (A^{\mathrm{hom}}_0)^{1/2}\bigr) \varepsilon^{q} (- d_x^2 + \varepsilon^{2}I)^{-q/2} F_\mathfrak{K} = [\varphi_0] \Phi^* [\cos\bigl(t \varepsilon^{-1} (b_0 k^2)^{1/2}\bigr) \varepsilon^{q} (k^2 + \varepsilon^{2})^{-q/2} \chi_\mathfrak{K}(k)] \Phi,
	\\
		\| (\Psi^* - [\varphi_0] \Phi^*) [\cos\bigl(t \varepsilon^{-1} (b_0 k^2)^{1/2}\bigr) \varepsilon^{q} (k^2 + \varepsilon^{2})^{-q/2} \chi_\mathfrak{K}(k)] \Phi \| \le C \varepsilon^{\min(1,q)}.
	\end{gather*}
	It remains to estimate the norm
	\begin{equation*}
		\Bigl\| \Psi^* [(\cos(t \varepsilon^{-1} (E(k) - \sigma_0)^{1/2}) - \cos(t \varepsilon^{-1} (b_0 k^2)^{1/2})) \varepsilon^{q} (k^2 + \varepsilon^{2})^{-q/2} \chi_\mathfrak{K}(k)] \Phi \Bigr\|_{L_2 (\mathbb{R}) \to L_2(\mathbb{R})}.
	\end{equation*}
	Transform the difference of the cosines using the identity
	\begin{multline*}
		\cos\bigl(t \varepsilon^{-1} (E(k) - \sigma_0)^{1/2}\bigr) - \cos\bigl(t \varepsilon^{-1} (b_0 k^2)^{1/2})\bigr)  \\
		= -2 \sin \left(\frac{1}{2} t \varepsilon^{-1}  \bigl((E(k) - \sigma_0)^{1/2} + (b_0 k^2)^{1/2}\bigr) \right) \sin \left(\frac{1}{2} t \varepsilon^{-1}  \bigl((E(k) - \sigma_0)^{1/2} - (b_0 k^2)^{1/2}\bigr) \right).
	\end{multline*}
	Obviously, $\left| \sin \left(\tfrac{1}{2} t \varepsilon^{-1} ((E(k) - \sigma_0)^{1/2} + (b_0 k^2)^{1/2}) \right)\right| \le 1$. Applying upper estimate~\eqref{lemma2_est_2} and taking into account~\eqref{sqrt_E_s_series_k=0}, we obtain that
	\begin{multline*}
		\Bigl\| \Psi^* [(\cos(t \varepsilon^{-1} (E(k) - \sigma_0)^{1/2}) - \cos(t \varepsilon^{-1} (b_0 k^2)^{1/2})) \varepsilon^{q} (k^2 + \varepsilon^{2})^{-q/2} \chi_\mathfrak{K}(k)] \Phi \Bigr\|_{L_2 (\mathbb{R}) \to L_2(\mathbb{R})} 
		\\
		\le	\frac{C\varepsilon^{2q/3} |t|^{q/3}}{(|\tilde{\gamma}_{k_0} (k_0)|^{-2/3} + \varepsilon^{4/3} |t|^{2/3})^{q/2}}, \qquad 0 < \varepsilon |t|^{-1} \le \tilde{\mathfrak{e}},
	\end{multline*}
	which completes the proof of~\eqref{cos_norm}.

	Now we turn to the proof of estimates~\eqref{sin_norm_r>0}, \eqref{sin_norm_r<0}.
	Firstly, we estimate the operator under the norm sign in~\eqref{sin_norm_reduction} multiplied by the projection $(I-F_\mathfrak{K})$ on the right:
	\begin{equation}
		\label{sin_(I-F)_est}
		\begin{aligned}
			&\begin{multlined}[c][0.8\textwidth]
				\varepsilon \Bigl( A_{\sigma_{0}}^{-1/2} \sin \bigl(t \varepsilon^{-1} A_{\sigma_{0}}^{1/2}\bigr) \Upsilon^{(+)}  \\ -  [\varphi_0] (A^{\mathrm{hom}}_0)^{-1/2} \sin (t \varepsilon^{-1} (A^{\mathrm{hom}}_0)^{1/2}) \Bigr) \varepsilon^{r} (- d_x^2 + \varepsilon^{2}I)^{-r/2}(I-F_\mathfrak{K})
			\end{multlined}
			\\
			&\begin{multlined}[c][0.8\textwidth]
				= \varepsilon \Bigl( \sum_{j=s}^\infty  \Psi_j^* R_{j-s+1} \bigl[ \bigl(E_j(k) - \sigma_0\bigr)^{-1/2} \sin \bigl(t \varepsilon^{-1} (E_j(k) - \sigma_0)^{1/2}\bigr) \chi_{\widetilde{\Omega}_{j-s+1}} (k) \bigr]   \\ -  [\varphi_0] \Phi^* \bigl[ (b_0 k^2)^{-1/2} \sin (t \varepsilon^{-1} (b_0 k^2)^{1/2}) \bigl] \Bigr) \bigl[ \varepsilon^{r} (k^2 + \varepsilon^{2})^{-r/2} (1-\chi_\mathfrak{K}(k)) \bigr] \Phi.
			\end{multlined}
		\end{aligned}
	\end{equation}
	We use the elementary inequality $|\sin x| \le 1$, $x \in \mathbb{R}$, and the estimates
	\begin{align*}
		\bigl(E_j(k) - \sigma_0\bigr)^{-1/2} \chi_{\widetilde{\Omega}_{j-s+1}} (k) (1-\chi_\mathfrak{K}(k)) &\le \bigl(E(\kappa) - \sigma_0\bigr)^{-1/2}, & &j\ge s,
		\\
		(b_0 k^2)^{-1/2} (1-\chi_\mathfrak{K}(k)) &\le b_0^{-1/2} \kappa^{-1}.
	\end{align*}
	Next, if $r \ge 0$, then $\varepsilon^{r} (k^2 + \varepsilon^2)^{-r/2} \le 1$, and the $(L_2(\mathbb{R})\to L_2(\mathbb{R}))$-norm of operator~\eqref{sin_(I-F)_est} is estimated by $C \varepsilon$. Consider the case where $-1 \le r < 0$. Let $N \ge 1$. Then
	\begin{align*}
		\varepsilon^{r} (k^2 + \varepsilon^2)^{-r/2} \chi_{[-N,N]}(k) &\le \varepsilon^{r} (N^2 + 1)^{-r/2}, & -1 &\le r < 0, 
		\\
		(b_0 k^2)^{-1/2} \varepsilon^{r} (k^2 + \varepsilon^2)^{-r/2} (1- \chi_{[-N,N]}(k))  &\le b_0^{-1/2} N^{-1} (N^2+1)^{-r/2} \varepsilon^{r}, & -1 &\le r < 0.
	\end{align*}
	Here we have taken into account that the function $(k^2 + \varepsilon^2)^{-r/2}$ is even and monotonically increasing in $k \ge 0$, and the function $|k|^{-1} (k^2 + \varepsilon^2)^{-r/2}$ is even and monotonically decreasing in $k \ge 1$.
	Next, from lower estimate~\eqref{E_l_estimates} and periodicity of the functions $E_j(k)$ we deduce
	\begin{equation*}
		E_j(k) \chi_{\widetilde{\Omega}_{j-s+1}} (k) \ge \alpha_0 \beta_0^2 \beta_1^{-2} (|k| + \pi (s-1))^2 \chi_{\widetilde{\Omega}_{j-s+1}} (k) \ge \alpha_0 \beta_0^2 \beta_1^{-2} k^2 \chi_{\widetilde{\Omega}_{j-s+1}} (k), \qquad j \ge s,
	\end{equation*}
	and therefore
	\begin{equation}
		\label{sin_(1-chi)_est}
		\begin{multlined}[c][0.9\textwidth]
			\bigl(E_j(k) - \sigma_0\bigr)^{-1/2} (1- \chi_{[-N,N]}(k)) \chi_{\widetilde{\Omega}_{j-s+1}} (k) \varepsilon^{r} (k^2 + \varepsilon^2)^{-r/2} \\ \le (\alpha_0 \beta_0^2 \beta_1^{-2} k^2 - \sigma_0)^{-1/2} (1- \chi_{[-N,N]}(k)) \varepsilon^{r} (k^2 + \varepsilon^2)^{-r/2} \le C \varepsilon^{r},\\ C = (\alpha_0 \beta_0^2 \beta_1^{-2} N^2 - \sigma_0)^{-1/2} (N^2 + 1)^{-r/2}, \qquad j \ge s,  \quad -1 \le r < 0. 
		\end{multlined}
	\end{equation}
	Here we have taken into account that the function in the middle part of~\eqref{sin_(1-chi)_est} monotonically decreases in $k \ge N$. The number $N$ is chosen such that $\alpha_0 \beta_0^2 \beta_1^{-2} N^2 - \sigma_0 > 0$. As a result, for $-1 \le r < 0$ operator~\eqref{sin_(I-F)_est} is estimated in the $(L_2(\mathbb{R}) \to L_2(\mathbb{R}))$-norm by $C \varepsilon^{r+1}$.
	
	Now, consider
	\begin{equation}
		\label{sin_F}
		\varepsilon \Bigl( A_{\sigma_{0}}^{-1/2} \sin \bigl(t \varepsilon^{-1} A_{\sigma_{0}}^{1/2}\bigr) \Upsilon^{(+)}  -  [\varphi_0] (A^{\mathrm{hom}}_0)^{-1/2} \sin (t \varepsilon^{-1} (A^{\mathrm{hom}}_0)^{1/2}) \Bigr) \varepsilon^r (- d_x^2 + \varepsilon^{2}I)^{-r/2}  F_\mathfrak{K}.
	\end{equation}
	Similarly to~\eqref{thrm_f1}, \eqref{thrm_f2}, we see that operator~\eqref{sin_F} equals
	\begin{equation*}
		\begin{multlined}[c][0.9\textwidth]
			\varepsilon \Bigl( \Psi^* \bigl[ \bigl(E(k) - \sigma_0\bigr)^{-1/2} \sin \bigl(t \varepsilon^{-1} (E(k) - \sigma_0\bigr)^{1/2}\bigr) \bigr]  -  [\varphi_0] \Phi^* \bigl[ (b_0 k^2)^{-1/2} \sin (t \varepsilon^{-1} (b_0 k^2)^{1/2}) \bigl] \Bigr)  \\ \times \bigl[  \varepsilon^r (k^2 + \varepsilon^{2})^{-r/2}  \chi_{\mathfrak{K}}(k)\bigr] \Phi.
		\end{multlined}
	\end{equation*}
	By~\eqref{sqrt_E_s_series_k=0}, \eqref{delta_fix2}, and \eqref{delta_fix3},
	\begin{multline*}
		\bigl| \bigl(E(k) - \sigma_0\bigr)^{-1/2} - (b_0 k^2)^{-1/2} \bigr| \chi_{\mathfrak{K}}(k) = \frac{|k|^3 |\tilde{\gamma}_0(k)|}{b_0^{1/2}|k| \bigl|b_0^{1/2}|k| + |k|^3 \tilde{\gamma}_0(k)\bigr|} \chi_{\mathfrak{K}}(k) 
		\\ \le \frac{3|k| |\tilde{\gamma}_0(0)|}{2 b_0^{1/2} \bigl|b_0^{1/2} + k^2 \tilde{\gamma}_0(k)\bigr|} \chi_{\mathfrak{K}}(k) \le 3 b_0^{-1} |k| |\tilde{\gamma}_0(0)| \chi_{\mathfrak{K}}(k),
	\end{multline*}
	whence
	\begin{equation*}
		\bigl| \bigl(E(k) - \sigma_0\bigr)^{-1/2} - (b_0 k^2)^{-1/2} \bigr| \varepsilon^{r+1} (k^2 + \varepsilon^{2})^{-r/2}  \chi_{\mathfrak{K}}(k) \le 3 b_0^{-1} |\tilde{\gamma}_0(0)| \varepsilon^{r+1} \kappa (\kappa^2 + \varepsilon^{2})^{-r/2}.
	\end{equation*}
	Here we have taken into account that the function $k (k^2+\varepsilon^2)^{-r/2}$ monotonically increases in  $k \ge 0$.
	
	Next, from~\eqref{lemma1_est_2} it directly follows that
	\begin{equation*}
		\varepsilon \| (\Psi^* - [\varphi_0] \Phi^*) [(b_0 k^2)^{-1/2} \sin (t \varepsilon^{-1} (b_0 k^2)^{1/2}) \varepsilon^{r} (k^2 + \varepsilon^{2})^{-r/2} \chi_\mathfrak{K}(k)] \Phi \| \le C \varepsilon^{\min(1,r+1)}.
	\end{equation*}
	It remains to estimate
		\begin{multline*}
			\varepsilon\Bigl\| \Psi^* \Bigl[ (b_0 k^2)^{-1/2} \Bigl(\sin\bigl(t \varepsilon^{-1} (E(k) - \sigma_0)^{1/2}\bigr) \\ -  \sin\bigl(t \varepsilon^{-1} (b_0 k^2)^{1/2}\bigr) \Bigr) \varepsilon^{r} (k^2 + \varepsilon^{2})^{-r/2} \chi_\mathfrak{K}(k)\Bigr] \Phi \Bigr\|_{L_2 (\mathbb{R}) \to L_2(\mathbb{R})}.
		\end{multline*}
	We have
	\begin{multline*}
		 \sin(t \varepsilon^{-1} (E(k) - \sigma_0)^{1/2}) - \sin(t \varepsilon^{-1} (b_0 k^2)^{1/2})\\=
		2 \cos \left(\tfrac{1}{2} t \varepsilon^{-1} ((E(k) - \sigma_0)^{1/2} + (b_0 k^2)^{1/2}) \right) \sin \left(\tfrac{1}{2} t \varepsilon^{-1} ((E(k) - \sigma_0)^{1/2} - (b_0 k^2)^{1/2})\right).
	\end{multline*}
	Obviously, $\left| \cos \left(\tfrac{1}{2} t \varepsilon^{-1} \bigl((E(k) - \sigma_0)^{1/2} + (b_0 k^2)^{1/2}\bigr) \right)\right| \le 1$. Application of upper estimate~\eqref{lemma2_est_3} together with~\eqref{sqrt_E_s_series_k=0} completes the proof of inequalities~\eqref{sin_norm_r>0}, \eqref{sin_norm_r<0}.	
\end{proof}

\subsection{The case where Condition~\ref{cond3} is fulfilled.}
Now it is convenient to realize points of the circle in~\eqref{Psi_l} in the following way:
\begin{equation*}
	\Psi_l \colon L_2(\mathbb{R}) \to L_2(\widetilde{\Omega}'_{s-l+1}), \qquad l = 1, \ldots, s,
\end{equation*}
where $\widetilde{\Omega}'_{j} = \widetilde{\Omega}_{j} + \pi$.

Consider the Cauchy problem for the nonstationary Schr\"{o}dinger equation
\begin{equation}
	\label{cond3_Cauchy_problem_Schrod}
	\left\{
	\begin{aligned}
		&i \frac{\partial}{\partial t} w_\varepsilon (x,t) = (A_\varepsilon w_\varepsilon)(x,t) ,\\
		&w_\varepsilon (x,0) = (\widetilde{\Upsilon}^{(-)}_\varepsilon f)(x),
	\end{aligned}
	\right.
\end{equation}
where
\begin{align*}
	&(\widetilde{\Upsilon}^{(-)}_\varepsilon f)(x) \coloneqq (2 \pi)^{-1/2} \int_\mathbb{R} (\Phi f) (k-\varepsilon^{-1} \pi) \sum_{j=1}^{s} e^{ikx} \varphi_j(x/\varepsilon, \varepsilon k) \chi_{\widetilde{\Omega}'_{s-j+1}} (\varepsilon k) \, dk,
	\\
	&\widetilde{\Upsilon}^{(-)}_\varepsilon f \in \mathcal{E}_{A_\varepsilon}[0, \varepsilon^{-2} \sigma_{\pi}] L_2(\mathbb{R}),
\end{align*}
and the corresponding "homogenized" problem
\begin{equation}
	\label{cond3_Cauchy_problem_Schrod_homog}
	\left\{
	\begin{aligned}
		&i \frac{\partial}{\partial t} w_0 (x,t) = -(A^{\mathrm{hom}}_{\pi} w_0) (x,t) ,\\
		&w_0 (x,0) = f (x).
	\end{aligned}
	\right.
\end{equation}
Here $A^{\mathrm{hom}}_\pi =  - b_\pi \frac{d^2}{d x^2}$, $\Dom A^{\mathrm{hom}}_\pi = H^2(\mathbb{R})$, is the operator acting in $L_2(\mathbb{R})$, which is called \emph{the effective operator at the right edge $\sigma_{\pi} = E(\pi)$ of the band $R(E)$}. 

The solutions of problems~\eqref{cond3_Cauchy_problem_Schrod} and~\eqref{cond3_Cauchy_problem_Schrod_homog} can be represented as follows:
\begin{equation}
	\label{w_represent}
	w_\varepsilon = e^{-i t A_\varepsilon} \widetilde{\Upsilon}^{(-)}_\varepsilon f, \qquad w_0 = e^{-i t (-A^{\mathrm{hom}}_{\pi})} f,
\end{equation}
where $\widetilde{\Upsilon}^{(-)}_\varepsilon \coloneqq \sum_{j=1}^s T_\varepsilon^* \Psi_j^* R'_{s-j+1}  \Phi T_\varepsilon [e^{i \pi x/\varepsilon}] $, and $R'_l$ is the operator of restriction to $\widetilde{\Omega}'_l$.

\begin{thrm}
	\label{cond3_Schrod_thrm}
	Suppose that Condition~\textup{\ref{cond3}} and~\eqref{gamma_ne_0} are fulfilled. Let $w_\varepsilon$ be the solution of problem~\eqref{cond3_Cauchy_problem_Schrod}, and let $w_0$ be the solution of problem~\eqref{cond3_Cauchy_problem_Schrod_homog}. Let $t \ne 0$, $f \in H^q(\mathbb{R})$, $0 \le q \le 2$. We have
	\begin{equation}
		\label{cond3_Schrod_thrm_est}
		\begin{multlined}[c][0.8\textwidth]
			\| w_\varepsilon(\cdot, t) - e^{-i t \varepsilon^{-2} \sigma_{\pi}} \varphi_{\pi}^{\varepsilon} e^{i \pi (\cdot/\varepsilon) } w_0(\cdot, t)\|_{L_2(\mathbb{R})} \le C (1+|t|^{q/4}) \varepsilon^{q/2} \|f\|_{H^q(\mathbb{R})}, 
			\\ 
			0 < \varepsilon |t|^{-1/2} \le \mathfrak{e},
		\end{multlined}
	\end{equation}
	with the constant $C = C(q, \kappa, \| \varphi_{\pi} \|_{L_\infty}, |\gamma_{\pi} (\pi)|, \|\theta_\pi\|_{M_\kappa(\pi)})$.
\end{thrm}
\begin{proof}
	By~\eqref{w_represent}, we have to prove the estimate
	\begin{equation*}
		\Bigl\| e^{-i t A_\varepsilon} \widetilde{\Upsilon}^{(-)}_\varepsilon - e^{-i t \varepsilon^{-2} \sigma_{\pi}} \varphi_{\pi}^{\varepsilon} [e^{i \pi \varepsilon^{-1} x}] e^{-i t (-A^{\mathrm{hom}}_{\pi})} \Bigr\|_{H^q (\mathbb{R}) \to L_2(\mathbb{R})} \le C (1+|t|^{q/4}) \varepsilon^{q/2},
	\end{equation*}
	where $t \ne 0$, $0 < \varepsilon |t|^{-1/2} \le \mathfrak{e}$, $0 \le q \le 2$. From~\eqref{Sobolev_isomorhism} and the unitarity of the scaling transformation it follows that
	\begin{multline*}
		\Bigl\| e^{-i t A_\varepsilon}  \widetilde{\Upsilon}^{(-)}_\varepsilon - e^{-i t \varepsilon^{-2} \sigma_{\pi}} \varphi_{\pi}^{\varepsilon} [e^{i \pi \varepsilon^{-1} x}] e^{-i t (-A^{\mathrm{hom}}_{\pi})} \Bigr\|_{H^q (\mathbb{R}) \to L_2(\mathbb{R})} \\ =
		\Bigl\| \Bigl( e^{-i t \varepsilon^{-2} A} \widetilde{\Upsilon}^{(-)}  - e^{-i t \varepsilon^{-2} \sigma_{\pi}} [\varphi_{\pi}] [e^{i \pi x}] e^{-i t \varepsilon^{-2} (-A^{\mathrm{hom}}_{\pi})} \Bigr)  \varepsilon^{q} (- d_x^2 + \varepsilon^{2}I)^{-q/2} \Bigr\|_{L_2 (\mathbb{R}) \to L_2(\mathbb{R})}.
	\end{multline*}
	Here $\widetilde{\Upsilon}^{(-)} = \sum_{j=1}^s \Psi_j^* R'_{s-j+1} \Phi [e^{i \pi x}]$. Next, 
	\begin{align*}
		\Phi^* [\varepsilon^q ((k-\pi)^2 + \varepsilon^2)^{-q/2} ]\Phi [e^{i\pi x}] &= [e^{i\pi x}] \varepsilon^{q} (- d_x^2 + \varepsilon^{2}I)^{-q/2},
		\\
		\Phi^* [e^{i t \varepsilon^{-2} b_{\pi} (k-\pi)^2} \varepsilon^q ((k-\pi)^2 + \varepsilon^2)^{-q/2} ]\Phi [e^{i\pi x}] &= [e^{i\pi x}] e^{-i t \varepsilon^{-2} (-A^{\mathrm{hom}}_{\pi})} \varepsilon^{q} (- d_x^2 + \varepsilon^{2}I)^{-q/2},
	\end{align*}
	whence
		\begin{multline*}
		\Bigl\| \Bigl( e^{-i t \varepsilon^{-2} A} \widetilde{\Upsilon}^{(-)} - e^{-i t \varepsilon^{-2} \sigma_{\pi}} [\varphi_{\pi}] [e^{i \pi x}] e^{-i t \varepsilon^{-2} (-A^{\mathrm{hom}}_{\pi})} \Bigr) \varepsilon^{q} (- d_x^2 + \varepsilon^{2}I)^{-q/2} \Bigr\|_{L_2 (\mathbb{R}) \to L_2(\mathbb{R})} \\ =
		\Bigl\| \Bigl( e^{-i t \varepsilon^{-2} A} \sum_{j=1}^s \Psi_j^* R'_{s-j+1} - e^{-i t \varepsilon^{-2} \sigma_{\pi}} [\varphi_{\pi}] \Phi^* [e^{i t \varepsilon^{-2} b_{\pi} (k-\pi)^2}] \Bigr) [\varepsilon^q ((k-\pi)^2 + \varepsilon^2)^{-q/2} ]\Phi  \Bigr\|_{L_2 (\mathbb{R}) \to L_2(\mathbb{R})}.
	\end{multline*}
	Obviously, $\varepsilon^q ((k-\pi)^2 + \varepsilon^2)^{-q/2} (1 - \chi_\mathfrak{K}(k-\pi)) \le \kappa^{-q} \varepsilon^q$, and so
	\begin{multline*}
		\Bigl\| \Bigl( e^{-i t \varepsilon^{-2} A} \sum_{j=1}^s \Psi_j^* R'_{s-j+1} - e^{-i t \varepsilon^{-2} \sigma_{\pi}} [\varphi_{\pi}] \Phi^* [e^{i t \varepsilon^{-2} b_{\pi} (k-\pi)^2}] \Bigr) \\ \times [\varepsilon^q ((k-\pi)^2 + \varepsilon^2)^{-q/2} (1 - \chi_\mathfrak{K}(k-\pi))]\Phi  \Bigr\|_{L_2 (\mathbb{R}) \to L_2(\mathbb{R})} \le (1 + \|\varphi_{\pi}\|_{L_\infty}) \kappa^{-q} \varepsilon^q.
	\end{multline*}
	Taking into account the equality
	\begin{multline*}
		e^{-i t \varepsilon^{-2} A} \sum_{j=1}^s \Psi_j^* R'_{s-j+1} [\varepsilon^q ((k-\pi)^2 + \varepsilon^2)^{-q/2} \chi_\mathfrak{K}(k-\pi)] \Phi \\ = \Psi^* [e^{-i t \varepsilon^{-2} E(k)} \varepsilon^{q} ((k-\pi)^2 + \varepsilon^{2})^{-q/2} \chi_\mathfrak{K}(k-\pi)] \Phi
	\end{multline*}
	and the estimate (which follows from~\eqref{lemma1_est_1})
	\begin{equation*}
		\label{thrm3_proof_lemma1_appl}
		\| (\Psi^* - [\varphi_{\pi}] \Phi^*) [e^{-i t \varepsilon^{-2} \sigma_{\pi}} e^{i t \varepsilon^{-2} b_{\pi} (k-\pi)^2} \varepsilon^{q} ((k-\pi)^2 + \varepsilon^{2})^{-q/2} \chi_\mathfrak{K}(k-\pi)] \Phi \| \le C \varepsilon^{\min(1, q)},
	\end{equation*}
	we conclude that it remains to estimate the norm
	\begin{equation*}
		\Bigl\| \Psi^*  [(e^{-i t \varepsilon^{-2} E(k)} - e^{-i t \varepsilon^{-2} \sigma_{\pi}} e^{i t \varepsilon^{-2} b_{\pi} (k-\pi)^2})  \varepsilon^{q} ((k-\pi)^2 + \varepsilon^{2})^{-q/2} \chi_\mathfrak{K}(k - \pi)] \Phi \Bigr\|_{L_2 (\mathbb{R}) \to L_2(\mathbb{R})}.
	\end{equation*} 
	By~\eqref{E_powerseries_2}, we have
	\begin{multline*}
		e^{-i t \varepsilon^{-2} E(k)} - e^{-i t \varepsilon^{-2} \sigma_{\pi}} e^{i t \varepsilon^{-2} b_{\pi} (k-\pi)^2} = e^{-i t \varepsilon^{-2} (\sigma_{\pi} - b_{\pi} (k-\pi)^2)} (e^{-i t \varepsilon^{-2} (k-\pi)^4 \gamma_{\pi} (k)} - 1)
		\\ 
		= -e^{-i t \varepsilon^{-2} (\sigma_{\pi} - b_{\pi} (k-\pi)^2)} e^{-\frac{1}{2} i t \varepsilon^{-2} (k-\pi)^4 \gamma_{\pi} (k)} \cdot 2i \sin\left( \frac{1}{2} t \varepsilon^{-2} (k-\pi)^4 \gamma_{\pi} (k) \right).
	\end{multline*}
	Application of estimate~\eqref{lemma2_est_1} together with $\bigl| e^{-i t \varepsilon^{-2} (\sigma_{\pi} - b_{\pi} (k-\pi)^2)} \bigr| = 1$, $\bigl| e^{-\frac{1}{2} i t \varepsilon^{-2} k^4 \gamma_\pi (k)}\bigr| = 1$ completes the proof.	
\end{proof}

Let $f, g \in L_2(\mathbb{R})$. Consider the Cauchy problem for the hyperbolic equation
\begin{equation}
	\label{cond3_Cauchy_problem_hyperb}
	\left\{
	\begin{aligned}
		&\frac{\partial^2}{\partial t^2} z_\varepsilon (x,t) =  (A_\varepsilon z_\varepsilon) (x,t) - \varepsilon^{-2} \sigma_{\pi} z_\varepsilon(x,t),\\
		&z_\varepsilon (x,0) = (\widetilde{\Upsilon}^{(-)}_\varepsilon f)(x), \quad (\partial_t z_\varepsilon) (x,0) = (\widetilde{\Upsilon}^{(-)}_\varepsilon g)(x),
	\end{aligned}
	\right.
\end{equation}
and the corresponding "homogenized" problem
\begin{equation}
	\label{cond3_Cauchy_problem_hyperb_homog}
	\left\{
	\begin{aligned}
		&\frac{\partial^2}{\partial t^2} z_0 (x,t) = - (A^{\textrm{hom}}_\pi z_0) (x,t) ,\\
		&z_0 (x,0) = f (x), \quad (\partial_t z_0) (x,0) = g (x).
	\end{aligned}
	\right.
\end{equation}
The solutions of problems~\eqref{cond3_Cauchy_problem_hyperb} and~\eqref{cond3_Cauchy_problem_hyperb_homog} satisfy the following representations:
\begin{align*}
		&z_\varepsilon = \cos \bigl(t A_{\sigma_\pi, \varepsilon}^{1/2}\bigr) \widetilde{\Upsilon}^{(-)}_\varepsilon f  + A_{\sigma_\pi, \varepsilon}^{-1/2} \sin \bigl(t A_{\sigma_\pi, \varepsilon}^{1/2}\bigr) \widetilde{\Upsilon}^{(-)}_\varepsilon g,
		\\
		&z_0 = \cos (t (A^{\mathrm{hom}}_\pi)^{1/2}) f + (A^{\mathrm{hom}}_\pi)^{-1/2} \sin (t (A^{\mathrm{hom}}_\pi)^{1/2}) g,
\end{align*}
where $A_{\sigma_\pi, \varepsilon} \coloneqq (\varepsilon^{-2}\sigma_\pi - A_\varepsilon) \mathcal{E}_{A_\varepsilon}[0, \varepsilon^{-2} \sigma_{\pi}]$. The following statement can be proved similarly to the proofs of Theorems~\ref{cond1_hyperb_thrm} and~\ref{cond3_Schrod_thrm}.
\begin{thrm}
	\label{cond3_hyperb_thrm}
	Suppose that Condition~\textup{\ref{cond3}} and~\eqref{gamma_ne_0} are fulfilled. Let $z_\varepsilon$ be the solution of problem~\eqref{cond3_Cauchy_problem_hyperb}, and let $z_0$ be the solution of problem~\eqref{cond3_Cauchy_problem_hyperb_homog}. Let $t \ne 0$, $f \in H^{q}(\mathbb{R})$, $g \in H^{r}(\mathbb{R})$. We have
	\begin{align}
		\label{cond3_hyperb_thrm_est}
		&\begin{multlined}[c][0.8\textwidth]
			\| z_\varepsilon(\cdot, t) - \varphi_{\pi}^{\varepsilon} e^{i \pi (\cdot/\varepsilon) }  z_0(\cdot, t)\|_{L_2(\mathbb{R})} \\ \le C \bigl( (1+|t|^{q/3}) \varepsilon^{2q/3} \|f\|_{H^{q}(\mathbb{R})} + (1+|t|^{(r+1)/3}) \varepsilon^{(2r+2)/3}  \|g\|_{H^{r}(\mathbb{R})} \bigr), \\ 0 < \varepsilon |t|^{-1} \le \tilde{\mathfrak{e}}, \quad 0 < \varepsilon \le 1, \quad 0 \le q \le 3/2, \quad 0 \le r \le 1/2,
		\end{multlined}
		\\
		\label{cond3_hyperb_thrm_est_r<0}
		&\begin{multlined}[c][0.8\textwidth]
			\| z_\varepsilon(\cdot, t) - \varphi_{\pi}^{\varepsilon} e^{i \pi (\cdot/\varepsilon) }  z_0(\cdot, t)\|_{L_2(\mathbb{R})} \\ \le C \bigl( (1+|t|^{q/3}) \varepsilon^{2q/3} \|f\|_{H^{q}(\mathbb{R})} + ((1+|t|^{(r+1)/3}) \varepsilon^{(2r+2)/3} + |t|^{1/3} \varepsilon^{2/3}) \|g\|_{H^{r}(\mathbb{R})} \bigr), \\ 0 < \varepsilon |t|^{-1} \le \tilde{\mathfrak{e}}, \quad 0 < \varepsilon \le 1, \quad 0 \le q \le 3/2, \quad -1 \le r < 0,
		\end{multlined}
	\end{align}
	with the constants $C = C(q, r, \kappa, \alpha_0, \beta_0, \beta_1, \sigma_\pi, b_\pi, \| \varphi_{\pi} \|_{L_\infty}, |\tilde{\gamma}_{\pi} (\pi)|, \|\theta_\pi\|_{M_\kappa(\pi)})$.
\end{thrm}

\subsection{The case where Condition~\ref{cond2} is fulfilled.} 
In $L_2(\mathbb{R})$, consider the operator $A^{\mathrm{hom}}_0 =  - b_0 \frac{d^2}{d x^2}$, $\Dom A^{\mathrm{hom}}_0 = H^2(\mathbb{R})$, which is called \emph{the effective operator at the right edge $\sigma_{0} = E(0)$ of the band $R(E)$}. Let $f, g \in L_2(\mathbb{R})$. Consider the following Cauchy problem for the nonstationary Schr\"{o}dinger equation and the corresponding "homogenized" problem:
\begin{equation}
	\label{cond2_Cauchy_problems_Schrod}
	\left\{
	\begin{aligned}
		&i \frac{\partial}{\partial t} \tilde{u}_\varepsilon (x,t) = (A_\varepsilon \tilde{u}_\varepsilon) (x,t) ,\\
		&\tilde{u}_\varepsilon (x,0) = (\Upsilon^{(-)}_\varepsilon f)(x),
	\end{aligned}
	\right. 
	\quad
	\left\{
	 \begin{aligned}
	 	&i \frac{\partial}{\partial t} \tilde{u}_0 (x,t) = -(A^{\mathrm{hom}}_0 \tilde{u}_0)(x,t),\\
	 	&\tilde{u}_0 (x,0) = f (x),
	 \end{aligned}
	 \right.
\end{equation} 
where
\begin{align*}
	&(\Upsilon^{(-)}_\varepsilon f)(x) \coloneqq (2 \pi)^{-1/2} \int_\mathbb{R} (\Phi f) (k) \sum_{j=1}^{s} e^{ikx} \varphi_j(x/\varepsilon, \varepsilon k) \chi_{\widetilde{\Omega}_{s-j+1}} (\varepsilon k) \, dk,
	\\
	&\Upsilon^{(-)}_\varepsilon f \in \mathcal{E}_{A_\varepsilon}[0, \varepsilon^{-2} \sigma_{0}] L_2(\mathbb{R}),
\end{align*}
and the similar Cauchy problems for the hyperbolic equations:
\begin{equation}
	\label{cond2_Cauchy_problems_hyperb}
	\left\{
	\begin{aligned}
		&\frac{\partial^2}{\partial t^2} \tilde{v}_\varepsilon (x,t) =  (A_\varepsilon \tilde{v}_\varepsilon)(x,t) - \varepsilon^{-2} \sigma_{0} \tilde{v}_\varepsilon(x,t),\\
		&\tilde{v}_\varepsilon (x,0) = (\Upsilon^{(-)}_\varepsilon f)(x), \\ &(\partial_t \tilde{v}_\varepsilon) (x,0) = (\Upsilon^{(-)}_\varepsilon g)(x),
	\end{aligned}
	\right.
	\quad
	\left\{
	\begin{aligned}
		&\frac{\partial^2}{\partial t^2} \tilde{v}_0 (x,t) = -(A^{\mathrm{hom}}_0 \tilde{v}_0)(x,t),\\
		&\tilde{v}_0 (x,0) = f (x), \\ &(\partial_t \tilde{v}_0) (x,0) = g (x).
	\end{aligned}
	\right.
\end{equation}
\begin{thrm}
	\label{cond2_Schrod_thrm}
	Suppose that Condition~\textup{\ref{cond2}} and~\eqref{gamma_ne_0} are fulfilled. Let $\tilde{u}_\varepsilon$, $\tilde{u}_0$ be the solutions of problems~\eqref{cond2_Cauchy_problems_Schrod}, $t \ne 0$, $f \in H^q(\mathbb{R})$, $0 \le q \le 2$. We have
	\begin{equation}
		\label{cond2_Schrod_thrm_est}
		\| \tilde{u}_\varepsilon(\cdot, t) - e^{-i t \varepsilon^{-2} \sigma_{0}} \varphi_{0}^{\varepsilon} \tilde{u}_0(\cdot, t)\|_{L_2(\mathbb{R})} \le C (1+|t|^{q/4}) \varepsilon^{q/2} \|f\|_{H^q(\mathbb{R})}, \qquad 0 < \varepsilon |t|^{-1/2} \le \mathfrak{e},
	\end{equation}
	with the constant $C = C(q, \kappa, \| \varphi_{0} \|_{L_\infty}, |\gamma_{0} (0)|, \|\theta_0\|_{M_\kappa(0)})$.
\end{thrm}
\begin{thrm}
	\label{cond2_hyperb_thrm}
	Suppose that Condition~\textup{\ref{cond2}} and~\eqref{gamma_ne_0} are fulfilled. Let $\tilde{v}_\varepsilon$, $\tilde{v}_0$ be the solutions of problems~\eqref{cond2_Cauchy_problems_hyperb}, and let $t \ne 0$, $f \in H^{q}(\mathbb{R})$, $g \in H^{r}(\mathbb{R})$. We have
	\begin{align}
		\label{cond2_hyperb_thrm_est}
		&\begin{multlined}[c][0.8\textwidth]
			\| \tilde{v}_\varepsilon(\cdot, t) - \varphi_0^{\varepsilon}  \tilde{v}_0(\cdot, t)\|_{L_2(\mathbb{R})} \\ \le C \bigl( (1+|t|^{q/3}) \varepsilon^{2q/3} \|f\|_{H^{q}(\mathbb{R})} + (1+|t|^{(r+1)/3}) \varepsilon^{(2r+2)/3}  \|g\|_{H^{r}(\mathbb{R})} \bigr), \\ 0 < \varepsilon |t|^{-1} \le \tilde{\mathfrak{e}}, \quad 0 < \varepsilon \le 1, \quad 0 \le q \le 3/2, \quad 0 \le r \le 1/2,
		\end{multlined}
		\\
		\label{cond2_hyperb_thrm_est_r<0}
		&\begin{multlined}[c][0.8\textwidth]
			\| \tilde{v}_\varepsilon(\cdot, t) - \varphi_0^{\varepsilon}  \tilde{v}_0(\cdot, t)\|_{L_2(\mathbb{R})} \\ \le C \bigl( (1+|t|^{q/3}) \varepsilon^{2q/3} \|f\|_{H^{q}(\mathbb{R})} + ((1+|t|^{(r+1)/3}) \varepsilon^{(2r+2)/3} + |t|^{1/3} \varepsilon^{2/3}) \|g\|_{H^{r}(\mathbb{R})} \bigr), \\ 0 < \varepsilon |t|^{-1} \le \tilde{\mathfrak{e}}, \quad 0 < \varepsilon \le 1, \quad 0 \le q \le 3/2, \quad -1 \le r < 0,
		\end{multlined}
	\end{align}
	with the constants $C = C(q, r, \kappa, \alpha_0, \beta_0, \beta_1, \sigma_{0}, b_0, \| \varphi_{0} \|_{L_\infty}, |\tilde{\gamma}_{0} (0)|, \|\theta_0\|_{M_\kappa(0)})$.
\end{thrm}

\subsection{The case where Condition~\ref{cond4} is fulfilled.} In $L_2(\mathbb{R})$, consider the operator $A^{\mathrm{hom}}_\pi =  - b_\pi \frac{d^2}{d x^2}$, $\Dom A^{\mathrm{hom}}_\pi = H^2(\mathbb{R})$, which is called \emph{the effective operator at the left edge $\sigma_{\pi} = E(\pi)$ of the band $R(E)$}. Let $f, g \in L_2(\mathbb{R})$. Consider the following Cauchy problem for the nonstationary Schr\"{o}dinger equation and the corresponding "homogenized" problem:
\begin{equation}
	\label{cond4_Cauchy_problems_Schrod}
	\left\{
	\begin{aligned}
		&i \frac{\partial}{\partial t} \tilde{w}_\varepsilon (x,t) = (A_\varepsilon \tilde{w}_\varepsilon) (x,t) ,\\
		&\tilde{w}_\varepsilon (x,0) = (\widetilde{\Upsilon}^{(+)}_\varepsilon f)(x),
	\end{aligned}
	\right.
	\quad
	\left\{
	\begin{aligned}
		&i \frac{\partial}{\partial t} \tilde{w}_0 (x,t) = (A^{\mathrm{hom}}_\pi \tilde{w}_0) (x,t) ,\\
		&\tilde{w}_0 (x,0) = f (x),
	\end{aligned}
	\right.	
\end{equation} 
where
\begin{align*}
	&(\widetilde{\Upsilon}^{(+)}_\varepsilon f)(x) \coloneqq (2 \pi)^{-1/2} \int_\mathbb{R} (\Phi f) (k - \varepsilon^{-1} \pi) \sum_{j=s}^{\infty} e^{ikx} \varphi_j(x/\varepsilon, \varepsilon k) \chi_{\widetilde{\Omega}'_{j-s+1}} (\varepsilon k) \, dk,
	\\
	&\widetilde{\Upsilon}^{(+)}_\varepsilon f \in \mathcal{E}_{A_\varepsilon}[\varepsilon^{-2} \sigma_{\pi}, \infty) L_2(\mathbb{R}),
\end{align*}
and the similar Cauchy problems for the hyperbolic equations:
\begin{equation}
	\label{cond4_Cauchy_problems_hyperb}
	\left\{
	\begin{aligned}
		&\frac{\partial^2}{\partial t^2}  \tilde{z}_\varepsilon (x,t) =  -(A_\varepsilon \tilde{z}_\varepsilon)(x,t) + \varepsilon^{-2} \sigma_\pi \tilde{z}_\varepsilon(x,t),\\
		&\tilde{z}_\varepsilon (x,0) = (\widetilde{\Upsilon}^{(+)}_\varepsilon f)(x), \\ &(\partial_t \tilde{z}_\varepsilon) (x,0) = (\widetilde{\Upsilon}^{(+)}_\varepsilon g)(x), 
	\end{aligned}
	\right.
	\quad
	\left\{
	\begin{aligned}
		&\frac{\partial^2}{\partial t^2} \tilde{z}_0 (x,t) = - (A^{\mathrm{hom}}_\pi \tilde{z}_0) (x,t) ,\\
		&\tilde{z}_0 (x,0) = f (x), \\ &(\partial_t \tilde{z}_0) (x,0) = g (x).
	\end{aligned}
	\right.
\end{equation}
\begin{thrm}
	\label{cond4_Schrod_thrm}	
	Suppose that Condition~\textup{\ref{cond4}} and~\eqref{gamma_ne_0} are fulfilled. Let $\tilde{w}_\varepsilon$, $\tilde{w}_0$ be the solutions of problems~\eqref{cond4_Cauchy_problems_Schrod}, and let $t \ne 0$, $f \in H^q(\mathbb{R})$, $0 \le q \le 2$. We have
	\begin{equation}
		\label{cond4_Schrod_thrm_est}
		\begin{multlined}[c][0.8\textwidth]
			\| \tilde{w}_\varepsilon(\cdot, t) - e^{-i t \varepsilon^{-2} \sigma_\pi} \varphi_{\pi}^{\varepsilon} e^{i \pi (\cdot/\varepsilon) } \tilde{w}_0(\cdot, t)\|_{L_2(\mathbb{R})} \le C (1+|t|^{q/4}) \varepsilon^{q/2} \|f\|_{H^q(\mathbb{R})}, \\ 0 < \varepsilon |t|^{-1/2} \le \mathfrak{e},
		\end{multlined}
	\end{equation}
	with the constant $C = C(q, \kappa, \| \varphi_{\pi} \|_{L_\infty}, |\gamma_{\pi} (\pi)|, \|\theta_\pi\|_{M_\kappa(\pi)})$.
\end{thrm}
\begin{thrm}
	\label{cond4_hyperb_thrm}
	Suppose that Condition~\textup{\ref{cond4}} and~\eqref{gamma_ne_0} are fulfilled. Let $\tilde{z}_\varepsilon$, $\tilde{z}_0$ be the solutions of problems~\eqref{cond4_Cauchy_problems_hyperb}, and let $t \ne 0$, $f \in H^{q}(\mathbb{R})$, $g \in H^{r}(\mathbb{R})$. We have
	\begin{align}
		\label{cond4_hyperb_thrm_est}
		&\begin{multlined}[c][0.8\textwidth]
			\| \tilde{z}_\varepsilon(\cdot, t) - \varphi_{\pi}^{\varepsilon} e^{i \pi (\cdot/\varepsilon) }  \tilde{z}_0(\cdot, t) \|_{L_2(\mathbb{R})} \\ \le C \bigl( (1+|t|^{q/3}) \varepsilon^{2q/3} \|f\|_{H^{q}(\mathbb{R})} + (1+|t|^{(r+1)/3}) \varepsilon^{(2r+2)/3}  \|g\|_{H^{r}(\mathbb{R})} \bigr), \\ 0 < \varepsilon |t|^{-1} \le \tilde{\mathfrak{e}}, \quad 0 < \varepsilon \le 1, \quad 0 \le q \le 3/2, \quad 0 \le r \le 1/2,
		\end{multlined}
		\\
		\label{cond4_hyperb_thrm_est_r<0}
		&\begin{multlined}[c][0.8\textwidth]
			\| \tilde{z}_\varepsilon(\cdot, t) - \varphi_{\pi}^{\varepsilon} e^{i \pi (\cdot/\varepsilon) }  \tilde{z}_0(\cdot, t)\|_{L_2(\mathbb{R})} \\ \le C \bigl( (1+|t|^{q/3}) \varepsilon^{2q/3} \|f\|_{H^{q}(\mathbb{R})} + ((1+|t|^{(r+1)/3}) \varepsilon^{(2r+2)/3} + |t|^{1/3} \varepsilon^{2/3}) \|g\|_{H^{r}(\mathbb{R})} \bigr), \\ 0 < \varepsilon |t|^{-1} \le \tilde{\mathfrak{e}}, \quad 0 < \varepsilon \le 1, \quad 0 \le q \le 3/2, \quad -1 \le r < 0,
		\end{multlined}
	\end{align}
	with the constants $C = C(q, r, \kappa, \alpha_0, \beta_0, \beta_1, \sigma_\pi, b_\pi, \| \varphi_{\pi} \|_{L_\infty}, |\tilde{\gamma}_{\pi} (\pi)|, \|\theta_\pi\|_{M_\kappa(\pi)})$.
\end{thrm}

\section{Concluding remarks}
\subsection{} Under the assumptions of Theorems~\ref{cond1_Schrod_thrm}, \ref{cond3_Schrod_thrm},  \ref{cond2_Schrod_thrm},  \ref{cond4_Schrod_thrm} for $q=2$ and under the assumptions of Theorems~\ref{cond1_hyperb_thrm}, \ref{cond3_hyperb_thrm}, \ref{cond2_hyperb_thrm},  \ref{cond4_hyperb_thrm}  for $q=3/2$, $r=1/2$ we have the error estimates of order $O((1+|t|^{1/2}) \varepsilon)$. The first power of $\varepsilon$ is order-sharp.  

\subsection{} Estimates~\eqref{cond1_hyperb_thrm_est_r<0}, \eqref{cond3_hyperb_thrm_est_r<0}, \eqref{cond2_hyperb_thrm_est_r<0}, \eqref{cond4_hyperb_thrm_est_r<0} are of interest, when the right-hand sides are small (i.e., when $\varepsilon |t|^{1/2}$ is small). If $\varepsilon |t|^{1/2} \le 1$, then the norms in the left-hand sides of~\eqref{cond1_hyperb_thrm_est_r<0}, \eqref{cond3_hyperb_thrm_est_r<0}, \eqref{cond2_hyperb_thrm_est_r<0}, \eqref{cond4_hyperb_thrm_est_r<0} are estimated by 
\begin{equation*}
	C \bigl( (1+|t|^{q/3}) \varepsilon^{2q/3} \|f\|_{H^{q}(\mathbb{R})} + ((1+|t|^{(r+1)/3}) \varepsilon^{(2r+2)/3}) \|g\|_{H^{r}(\mathbb{R})} \bigr)
\end{equation*}
under the same assumptions.

\subsection{} With the help of the lower estimates formulated in Lemma~\ref{lemma2}, it can be proved that the results of Theorems~\ref{cond1_Schrod_thrm}--\ref{cond4_hyperb_thrm} are sharp with respect to the norm type as well as with respect to the dependence on $t$ (for large $t$). Let us show that for the example of the result of Theorem~\ref{cond1_Schrod_thrm} (for $q=2$).
\begin{thrm}
	\label{cond1_Schr_thrm_shrp_thrm}
	Suppose that Condition~\textup{\ref{cond1}} and~\eqref{gamma_ne_0} are fulfilled. Let $u_\varepsilon$ be the solution of problem~\eqref{cond1_Cauchy_problem_Schrod}, and let $u_0$ be the solution of problem~\eqref{cond1_Cauchy_problem_Schrod_homog}. Let $t \ne 0$ and $0 \le q' < 2$. Then there does not exist a constant $\mathcal{C}(t)>0$ such that the estimate
	\begin{equation}
			\label{cond1_Schr_thrm_shrp_est}
			\| u_\varepsilon(\cdot, t) - e^{-i t \varepsilon^{-2} \sigma_0} \varphi_0^{\varepsilon}  u_0(\cdot, t)\|_{L_2(\mathbb{R})} \le \mathcal{C}(t) \varepsilon \|f\|_{H^{q'}(\mathbb{R})}
	\end{equation}
	holds for all sufficiently small $\varepsilon >0$.
\end{thrm}
\begin{proof}
	We prove by contradiction. Without loss of generality, it suffices to assume that $1 \le q' < 2$. Suppose that~\eqref{cond1_Schr_thrm_shrp_est} is valid. By~\eqref{u_represent}, it means that the estimate
	\begin{equation*}
		\Bigl\| e^{-i t A_\varepsilon} \Upsilon^{(+)}_\varepsilon - e^{-i t \varepsilon^{-2} \sigma_0} [\varphi_0^{\varepsilon}] e^{-i t A^{\mathrm{hom}}_0} \Bigr\|_{H^{q'} (\mathbb{R}) \to L_2(\mathbb{R})} \le \mathcal{C}(t) \varepsilon
	\end{equation*}
	holds for all sufficiently small $\varepsilon$. 
	By virtue of~\eqref{Sobolev_isomorhism}--\eqref{exp-exp_sin} (with $q$ replaced by $q'$), we conclude that 
	\begin{multline*}
		\left\| \Psi^* \left[2i e^{-i t \varepsilon^{-2} (\sigma_0 + b_0 k^2 + (1/2) k^4 \gamma_0 (k))} \sin\left( \frac{1}{2} t \varepsilon^{-2} k^4 \gamma_0 (k) \right)  \varepsilon^{q'} (k^2 + \varepsilon^{2})^{-q'/2} \chi_\mathfrak{K}(k)\right] \Phi \right\|_{L_2 (\mathbb{R}) \to L_2(\mathbb{R})} \\  \le \check{\mathcal{C}}(t) \varepsilon
	\end{multline*}
	with a constant $\check{\mathcal{C}}(t) > 0$ for all sufficiently small $\varepsilon$. Note that the initial space of the operator $\Psi^*$ coincides with $\Ran \Psi = L_2(\widetilde{\Omega})$. Obviously, the range of the operator $$\left[2i e^{-i t \varepsilon^{-2} (\sigma_0 + b_0 k^2 + (1/2) k^4 \gamma_0 (k))} \sin\left( \frac{1}{2} t \varepsilon^{-2} k^4 \gamma_0 (k) \right)  \varepsilon^{q'} (k^2 + \varepsilon^{2})^{-q'/2} \chi_\mathfrak{K}(k)\right] \Phi$$ is contained in the initial space of $\Psi^*$. Together with lower estimate~\eqref{lemma2_est_1}, this yields
	\begin{equation*}
		\label{cond1_Schr_thrm_shrp_proof_f1}
		\frac{\varepsilon^{q'/2} |t|^{q'/4}}{(|\gamma_0 (0)|^{-1/2} + \varepsilon |t|^{1/2})^{q'/2}} \le \check{\mathcal{C}}(t) \varepsilon
	\end{equation*}
	for all sufficiently small $\varepsilon$. But this is not true if $q' < 2$. This contradiction completes the proof.
\end{proof}

\begin{thrm}
	Suppose that Condition~\textup{\ref{cond1}} and~\eqref{gamma_ne_0} are fulfilled. Let $u_\varepsilon$ be the solution of problem~\eqref{cond1_Cauchy_problem_Schrod}, and let $u_0$ be the solution of problem~\eqref{cond1_Cauchy_problem_Schrod_homog}. Let $q' \ge 2$. Then there does not exist a function $\mathcal{C}(t) > 0$ such that $\lim_{t \to \infty} \mathcal{C}(t)/|t|^{1/2} = 0$ and estimate~\eqref{cond1_Schr_thrm_shrp_est} holds for all $t \in \mathbb{R}$ and all sufficiently small $\varepsilon > 0$. 
\end{thrm}
\begin{proof} We also prove by contradiction. Repeating the same arguments that were used in the proof of Theorem~\ref{cond1_Schr_thrm_shrp_thrm}, we obtain the estimates
	\begin{alignat*}{2}
		\frac{\varepsilon^{q'/2} |t|^{q'/4}}{(|\gamma_0 (0)|^{-1/2} + \varepsilon |t|^{1/2})^{q'/2}} &\le \check{\mathcal{C}}(t) \varepsilon, & \qquad &\text{for} \quad 2 \le q' \le 4,
		\\
		|\gamma_{0} (0)| \varepsilon^2 |t| &\le \check{\mathcal{C}}(t) \varepsilon, & &\text{for} \quad q' > 4,
	\end{alignat*}
	for $t \ne 0$ and all sufficiently small $\varepsilon$, which can be rewritten as
	\begin{alignat*}{2}
		\frac{\varepsilon^{q'/2-1} |t|^{q'/4-1/2}}{(|\gamma_0 (0)|^{-1/2} + \varepsilon |t|^{1/2})^{q'/2}} &\le \frac{\check{\mathcal{C}}(t)}{|t|^{1/2}}, & \qquad &\text{for} \quad 2 \le q' \le 4,
		\\
		|\gamma_{0} (0)| \varepsilon |t|^{1/2} &\le \frac{\check{\mathcal{C}}(t)}{|t|^{1/2}}, & &\text{for} \quad q' > 4,
	\end{alignat*}
where $\check{\mathcal{C}}(t)$ is a positive function such that $\lim_{t \to \infty} \check{\mathcal{C}}(t)/ |t|^{1/2} = 0$.  But this estimate is not true for large $|t|$ and sufficiently small $\varepsilon = |t|^{-1/2}$. This contradiction completes the proof.
\end{proof}

\subsection{} Finally, throughout the paper we have assumed that~\eqref{gamma_ne_0} is fulfilled. This is the generic case. Now suppose that
\begin{equation*}
	\gamma_{k_0} (k) = (k - k_0)^{2m} \check{\gamma}_{k_0} (k), \qquad \check{\gamma}_{k_0} (k_0) \ne 0, \quad m \ge 1.
\end{equation*} 
Then the statements of Theorems~\ref{cond1_Schrod_thrm}--\ref{cond4_hyperb_thrm} can be improved. In particular, for $f \in H^{q_1}(\mathbb{R})$, $q_1 = (m+2)/(m+1)$, the norms in the left-hand sides of~\eqref{cond1_Schrod_thrm_est}, \eqref{cond3_Schrod_thrm_est}, \eqref{cond2_Schrod_thrm_est}, \eqref{cond4_Schrod_thrm_est} are estimated by
\begin{equation*}
	C (1 + |t|^{1/(2m+2)}) \varepsilon \|f\|_{H^{q_1}(\mathbb{R})};
\end{equation*}
and for $f \in H^{q_2}(\mathbb{R})$, $g \in H^{q_3}(\mathbb{R})$,  $q_2 = (2m+3)/(2m+2)$, $q_3 = 1/(2m+2)$, the norms in the left-hand sides of~\eqref{cond1_hyperb_thrm_est}, \eqref{cond3_hyperb_thrm_est}, \eqref{cond2_hyperb_thrm_est}, \eqref{cond4_hyperb_thrm_est} are estimated by
\begin{equation*}
	C (1 + |t|^{1/(2m+2)}) \varepsilon (\|f\|_{H^{q_2}(\mathbb{R})} + \|g\|_{H^{q_3}(\mathbb{R})}).
\end{equation*}

\end{document}